\documentclass[12pt]{amsart}

\usepackage{CJK}

\newtheorem{theorem}{Theorem}[section]
\newtheorem{lemma}[theorem]{Lemma}
\newtheorem{proposition}[theorem]{Proposition}
\newtheorem{corollary}[theorem]{Corollary}

 \theoremstyle{definition}

\theoremstyle{remark}
\newtheorem{remark}[theorem]{Remark}

\numberwithin{equation}{section}

\begin{document}
\begin{CJK*}{GBK}{song}

%\title[Degenerate fully nonlinear elliptic equations]
\title[The Dirichlet problem for degenerate equations]
%\title[Degenerate equations on Riemannian manifolds of mean concave boundary]
{The Dirichlet problem for a class of degenerate fully nonlinear elliptic equations on Riemannian manifolds with mean concave boundary}
%{Quantitative boundary estimates for a class of degenerate fully non-linear elliptic equations on Riemannian manifolds with  mean concave boundary}
%{The Dirichlet problem for a class of degenerate fully nonlinear elliptic equations on Riemannian manifolds with nonpositive mean curvature boundary}
%{The Dirichlet problem for a class of degenerate  fully nonlinear elliptic equations on Riemannian manifolds with boundary of nonpositive mean curvature}
%{The Dirichlet problem for a class of degenerate  fully nonlinear elliptic equations on Riemannian manifolds with boundary of nonpositive mean curvature}
%{Regularity of a class of degenerate  fully nonlinear elliptic equations on Riemannian manifolds with boundary of nonpositive mean curvature}
%{Regularity of degenerate  fully nonlinear elliptic equations on Riemannian manifolds with \textit{$1$-concave} boundary}

 \author{Ri-Rong Yuan}
%\address{School of Mathematics (Zhuhai), Sun Yat-sen University, Zhuhai 519082, China}
%\email{yuanrr@mail.sysu.edu.cn; rirongyuan@stu.xmu.edu.cn}
\email{rirongyuan@stu.xmu.edu.cn}

\thanks{%Research supported in part by NSFC (Grant No. 11801587)}
%The  author is  supported by National Natural Science Foundation of China (Grant No. 11801587)}
%The author is partially supported by NSFC (Grant No. 11801587)}
% Research supported in part by NSF in China, Grant No. 11801587.
% Research of the author appreciates the support and hospitality of the Department, and also wants
% to thank Graduate School of Xiamen University for sponsoring his visit to OSU through the Short-term Research Program.
%The author is  supported by NSFC grant 11801587.
 The author is supported by NSF in China, grant 11801587.
 }
%https://orcid.org/0000-0001-6282-4732

\date{}

\begin{abstract}

This article studies the Dirichlet problem for a class of degenerate fully nonlinear elliptic equations on Riemannian manifolds with \textit{mean concave} boundary in the sense that the mean curvature of the boundary is \textit{nonpositive}.
 The proof is primarily based on a quantitative boundary estimate.
  Also, we obtain analogous results in complex variables.
   In Appendix,  the subsolutions are also constructed on certain topologically product manifolds.

%{\em Mathematical Subject Classification (2010):}
%{\em Mathematics Subject Classification (2010): 35J15, 35J70, 58J05, 35B45.}


\end{abstract}

 \maketitle

 %\tableofcontents

 \section{Introduction}

Let $(M,g)$ be an $n$-dimensional ($n\geq3$) compact Riemannian manifold with \textit{mean concave} boundary $\partial M$ and the Levi-Civita connection $\nabla$,   $\bar M:=M\cup\partial M$.
Let $A$ be a smoothly symmetric $(0,2)$-type tensor,
$\eta=\eta_{ij}^k \frac{\partial}{\partial x^k}\otimes d x^i\otimes dx^j$ be a smooth $(1,2)$-type tensor  with $\eta^k_{ij}=\eta^k_{ji}$,
$Z(du)=\mathrm{tr}_g( W(du))g-W(du),$
 where $(W(du))_{ij}=\sum_{k=1}^n u_k\eta_{ij}^k$.
A  further technical hypothesis is needed:
 For each $p_0\in\partial M$,
 under local coordinates  with $g_{ij}(p_0)=\delta_{ij}$,  $\sum_{i\neq k}\eta^{k}_{ii}(p_0)=0$ for each $k$.
Typical examples are $\eta\equiv0$ and more general $\eta_{ij}^k=\delta_{ik}\zeta_j+\delta_{jk}\zeta_i$ for a smooth $(0,1)$ tensor $\zeta=\zeta_i dx^i$.
%   \begin{equation}
%   \label{technique}
%   \begin{aligned}
%   \sum_{i\neq k}\eta^{k}_{ii}=0.
%   \end{aligned}
%   \end{equation}

\vspace{1mm}
 In this paper, we propose a new hypothesis   on the boundary  in an attempt to
  establish a quantitative boundary estimate
 and  solve the following
fully nonlinear elliptic equation possibly with degenerate right-hand side %generated by $f$ of eigenvalues.
\begin{equation}
\label{mainequ0}
\begin{aligned}
 f(\lambda(U[u])) =\psi \mbox{ in } M, %\mbox{ where } U[u]=A+ (\Delta u)g-\nabla^2 u +Z(du),
\end{aligned}
\end{equation}
%with prescribed boundary value
\begin{equation}
\label{mainequ1}
\begin{aligned}
u=  \varphi    \mbox{ on }\partial M,
\end{aligned}
\end{equation}
where  $\lambda(U[u])$ denote the eigenvalues of $U[u]=A+ (\Delta u)g-\nabla^2 u +Z(du)$
  with respect to $g$,
$\nabla^2u$ is real Hessian
with $\nabla^2u(X,Y)=\nabla_{XY}u:=YXu-(\nabla_{Y}X)u \mbox{ for }X,Y\in TM,$
 $\Delta=\mathrm{tr}_g\nabla^2$ is the  Laplacian operator.
  Moreover, $\varphi$ and $\psi$ are sufficiently smooth functions with %an extension
 ${w}|_{\partial M}=\varphi$ for some \textit{admissible} function
 ${w}\in C^2(\bar M)$ defined as $ \lambda(U[{w}])\in \Gamma,$   also with $\sup_M\psi<\sup_{\Gamma}f$ that is automatically
 satisfied if %there exists a strictly admissible subsolution obeying \eqref{existenceofsubsolution-de}
   there exists an admissible subsolution satisfying \eqref{existenceofsubsolution} below.
   As in \cite{CNS3},  $\Gamma$ is an open symmetric and convex cone
with vertex at the origin, $\Gamma_n\subseteq\Gamma\subset \Gamma_1$ and boundary
$\partial \Gamma\neq \emptyset$, on which $f$ is a smooth and
  symmetric function satisfying
%the following  fundamental conditions: %(going back at least to the work  of  \cite{CNS3}): % concerning
 % which studies Dirichlet problem on bounded domains of $\mathbb{R}^n$,
\begin{equation}
\label{elliptic}
 f_{i}(\lambda):= %f_{\lambda_i}(\lambda)=
 \frac{\partial f}{\partial \lambda_{i}}(\lambda)> 0  \mbox{ in } \Gamma,\  1\leq i\leq n,
\end{equation}
\begin{equation}
\label{concave}
 f \mbox{ is  concave in } \Gamma.
\end{equation}
   where $\Gamma_1:=\{\lambda\in \mathbb{R}^{n}:   \sum_{i=1}^n\lambda_i>0\}$ and
 $\Gamma_{n} :=\{\lambda\in \mathbb{R}^{n}:   \mbox{ each } \lambda_i>0\}.$ % \mbox{ }
%and $\sigma_k(\lambda)=\sum_{1\leq i_1<i_2<\cdots<i_k\leq n} \lambda_{i_1}\cdots \lambda_{i_k}.$
In order to study equation  \eqref{mainequ0} %and Dirichlet problem \eqref{mainequ}-\eqref{mainequ1}
    within the framework of elliptic equations,  we shall %look for
    seek solutions
in the class of $C^2$-\textit{admissible} functions.

\vspace{1mm}
The study of such equations
goes back at least to \cite{CNS3,Ivochkina1981},
 and  since then it has been carried widely out %done
 in numerous literature
(cf. \cite{Guan1994The,Guan1998CVPDE,Guan12a,Guan14,Chu2020Jiao,LiYY1990,Gabor,GQY2018,Trudinger95,Urbas2002,yuan2017}), particularly
\cite{GTW15,Tosatti2017Weinkove,Tosatti2013Weinkove,Popovici2015,guan-nie,yuan2019} on
%\cite{FuWangWuFormtype2010,FuWangWuFormtype2015,GTW15,Tosatti2017Weinkove,Tosatti2013Weinkove,Popovici2015,guan-nie,yuan2019} on
Monge-Amp\`ere equation for $(n-1)$-plurisubharmonic functions %and  Form-type Calabi-Yau equation
 from non-K\"ahler geometry.
%while in complex setting, this type of equations involving $\Delta u\omega-\sqrt{-1}\partial\overline{\partial}u$.
Specifically, due to the %increasing
interests and problems from differential geometry and %the analysis of
PDEs (cf. \cite{GuanP2001ICCM} and references therein),
% (cf. \cite{Chen,Donaldson2010Nahm,Guan1998The,GuanP2002The,Guan2009Zhang,Gursky2018Streets} and references therein),
%\cite{Donaldson99,Mabuchi87,Semmes92,Chen,Guan1998The,Chern-Levine-Nirenberg,Bedford1979Taylor,GuanP2002The,Guan2009Zhang,Gursky2018Streets,Donaldson2010Nahm},
 %\cite{Chen,Guan1998The,Phong-Sturm2010,GuanP2002The,Guan2009Zhang}
 it would be important to understand the solutions %to  fully nonlinear elliptic
%  partial differential equations %with degenerate
 whenever the right-hand side is degenerate
 $$\inf_M\psi=\sup_{\partial\Gamma}f,$$
where  $\sup_{\partial \Gamma}f :=\sup_{\lambda_{0}\in \partial \Gamma } \limsup_{\lambda\rightarrow \lambda_{0}}f(\lambda).$
 There are some papers
 should be mentioned (cf.  \cite{CNS-deg,GuanP1997Duke,GTW1999Acta,ITW2004,Trudinger84Urbas}, see also \cite{GuanP2001ICCM} for some relative open problems),
which concern %includes among others
degenerate  real Monge-Amp\`ere equation and more general degenerate
fully nonlinear elliptic equations %among in particular degenerate real Monge-Amp\`ere equation,
on $\Omega\subset\mathbb{R}^n$.
% On the other hand, when the background space is a curved Riemannian manifold,
 %to the best of the author's knowledge, there is few known work on this topic.
However, it is still less known for general cases
 %both of equations and of background spaces.
 of both equations and background spaces.
  On the background space being a curved Riemannian manifold,
\renewcommand{\thefootnote}{\fnsymbol{footnote}}with the assumption that %there is a strictly admissible subsolution, and
the boundary is \textit{concave}\footnote{We call $\partial M$ is \textit{concave}, if the second
 fundamental form with respect to  %outer normal vector along the boundary
  $-\nu$,
 denoted by $\mathrm{II}_{\partial M}$, %is nonpositive
 is negative semidefinite, where and hereafter %$$\nu \mbox{ is the unit inner normal vector along the boundary.}$$
  $$\nu  \textit{ is the unit inner normal vector along boundary.}$$
 Here are examples of such spaces: Let $\bar M=X\times [0,1]$ ($\partial X=\emptyset$) equip with $g=e^\phi g_X+ dx^n\otimes dx^n$ where    $\phi$  is a  smooth function  (on  $\bar M$)
  with $\nabla_\nu \phi|_{\partial M}\geq0$,
     while on such %topologically product spaces,
   warped product spaces subsolutions  for certain equations
    are  constructed in  Appendix \ref{appendix}.
    } and there is a strictly admissible subsolution,   the author \cite{yuan2017,yuan2019}
  % (see  Theorem \ref{mainthm1-20172019} below, while the papers primarily treat %fully nonlinear elliptic
% equations on complex manifolds)
   %Then one can verify
 %the warped product spaces
%such spaces admit  \textit{concave} boundary if $\nabla_\nu h|_{\partial M}\geq0$.
% as defined in Theorem \ref{proposition-quar-yuan1-thm} $\nu$ is the unit inner normal vector along the boundary.}
 %Let $M=X\times [0,1]$ be with $g=hg_X+\rho dx^n\otimes dx^n$ where $h, \rho$ are smooth functions on $M$. Then one can verify
 %the warped product spaces
%such spaces admit  \textit{concave} boundary if $\nabla_\nu h|_{\partial M}\geq0$.
% as defined in Theorem \ref{proposition-quar-yuan1-thm} $\nu$ is the unit inner normal vector along the boundary.}
% and there exist strictly admissible subsolutions,   the author \cite{yuan2017,yuan2019}
 % (see Theorem \ref{mainthm1-20172019} below)
% establishes a quantitative boundary estimate
%to derive the gradient bound by using a blow-up argument, and then
 solved the Dirichlet problem for  %a general class of
 the following degenerate fully nonlinear elliptic
  equations (see  Theorem \ref{mainthm1-20172019} below, while the papers primarily treat %fully nonlinear elliptic
   equations on complex manifolds)
   \begin{equation}
\label{mainequ20172019}
\begin{aligned}
F(\mathfrak{g}[u]):=f(\lambda(\mathfrak{g}[u]))=\,& \psi, %\mbox{ in } M  % \mbox{  }  \mbox{  } \\
%u=\,& \varphi \mbox{ on } \partial M
%\mbox{  } \mathfrak{g}[u]=\chi+\nabla^2u+d u\otimes\eta+\eta\otimes d u,
\end{aligned}
\end{equation}
in which $\mathfrak{g}[u]$$=\chi+\nabla^2u+W(du)$,
% $W(du)=d u\otimes\eta+\eta\otimes d u$
 $\chi$ is a smoothly symmetric  $(0,2)$  tensor, and % $f$ further obeys
  $f$ further satisfies
 % (see the discussion presented in the final section of \cite{Gabor}).
%as in \cite{yuan2017,yuan2019},
 \begin{equation}
\label{addistruc}
\begin{aligned}
\mbox{For each $\sigma<\sup_{\Gamma}f$ and } \lambda\in \Gamma, \mbox{   } \lim_{t\rightarrow +\infty}f(t\lambda)>\sigma,
\end{aligned}
\end{equation}
or equivalently  (according  to the Lemma 3.4 of \cite{yuan2019})
      \begin{equation}
\label{charact-yuan}
\begin{aligned}
\sum_{i=1}^n f_i(\lambda)\mu_i>0, \mbox{  } \forall \lambda, \mu\in \Gamma,
     \end{aligned}
\end{equation}
%according  to the Lemma 3.4 of \cite{yuan2019},
%according  to the Lemma 3.4 of \cite{yuan2019} on a criterion of $f$ satisfying \eqref{elliptic}, \eqref{concave} and \eqref{addistruc},
%  In particular $\sum_{i=1}^n f_i(\lambda)\lambda_i>0.$
which is satisfied  by %many symmetric functions %rising from differential geometry and fully nonlinear equations,  including
  the functions with $\Gamma=\Gamma_n$, and also by homogenous functions of degree one with  $f|_{\Gamma}>0$. %as special cases.
  %    We also refer to the Lemma 3.4 of \cite{yuan2019} for a
 %    criterion of $f$ satisfying \eqref{elliptic}, \eqref{concave} and \eqref{addistruc}, which
    % and also to \cite{yuan2019Kahlercone} for follow-up work concerning a class of degenerate
   % fully nonlinear equations on K\"ahler cones related to Sasaki geometry.
  %(See the final section of \cite{yuan2017} for $\eta=0$ and Remark 7.5 of \cite{yuan2019} for the announcement on general $\eta$,  see also Theorem \ref{mainthm1-20172019} below;  the papers primarily treat equations on complex manifolds).
%In complex setting degenerate fully nonlinear equations on complex manifold and K\"ahler cones related to Sasaki geometry has been settled by the author in \cite{yuan2017,yuan2019} and follow-up work \cite{yuan2019Kahlercone}.

\vspace{1mm}
%For the equations being of the form \eqref{mainequ0}, %we can relax in this article the \textit{concave} assumption to the boundary
%we %show in this article that the Dirichlet problem is solvable
In the following theorem
we show %a new and
an interesting approach to %the study of
 degenerate equations of the form
 \eqref{mainequ0}, when $(M,g)$ supposes \textit{mean concave} boundary %(i.e. the mean curvature of boundary, denoted by $\mathrm{H}_{\partial M}$, is nonpositive)
which %for $n\geq 3$
 includes among others
  Riemannian manifold whose boundary is a minimal hypersurface.
    Such a hypothesis %on boundary
    is reasonable   %to make on boundary
 according to   significant progress
 on Yau's conjecture   (cf. \cite{I-M-N2018,Li.Y-2019,A.Song2018} and references therein).
 % show there are infinitely many closed minimal hypersurfaces in a  closed Riemannian manifold.
 %and is clearly much more general than those Riemannian manifolds with \textit{concave} boundary,

%\vspace{1mm}
%The main theorems are as follows.
%To be precise, we state it as follows.
\begin{theorem}
%Let $\partial M$ be smooth and \textit{$1$-concave},
\label{mainthm1-de}
Let $(M,g)$ be a compact Riemannian manifold with smooth \textit{mean concave}  boundary.
In addition to \eqref{elliptic}, \eqref{concave}, \eqref{addistruc} and
 $f\in C^\infty(\Gamma)\cap C(\overline{\Gamma})$,
 we suppose %the given data
 $\varphi\in C^{2,1}(\partial M)$,
 $\psi\in C^{1,1}(\bar M)$, $\inf_M\psi=\sup_{\partial \Gamma}f$. Then Dirichlet problem \eqref{mainequ0}-\eqref{mainequ1} admits a  (weak)
   solution $u\in C^{1,1}(\bar M)$ with $\lambda(U[u])\in \overline{\Gamma}$
   and $\Delta u\in L^\infty(\bar M)$, provided that
   there is  a strictly admissible subsolution $\underline{u}\in C^{2,1}(\bar M)$ obeying
 % for some $\delta_0>0$ there holds
 \begin{equation}
\label{existenceofsubsolution-de}
\begin{aligned}
f(\lambda(U[\underline{u}])) \geq  \psi +\delta_0  \mbox{ in } \bar M, \mbox{   }
\underline{u}= \varphi   \mbox{ on } \partial  M,
\end{aligned}
\end{equation}
 for some $\delta_0>0$.
Moreover, if  the mean curvature of boundary, say $\mathrm{H}_{\partial M}$, is strictly negative then the above statement still holds for
$\partial M\in C^{2,1}$.
\end{theorem}
%It seems interesing to make such a hypothesis on boundary since the deep results in \cite{I-M-N2018,M-N-S2019,Li.Y-2019} show that there exist infinite many minimal hypersurfaces on a closed generic Riemannian manifold.
Consequently, we have a corollary. %Analogous results in complex counterpart can be obtained by using certain results from \cite{yuan2017,yuan2019}.
\begin{corollary}
\label{minimal-thm}
Let $(M,g)$ be a closed connected Riemannian manifold, $\Sigma\subset  M$ be a closed connected smoothly minimal hypersurface in $(M,g)$ such that $M\setminus \Sigma=M_1\cup M_2$, $M_1\cap M_2=\emptyset$ and $\partial M_1=\partial M_2=\Sigma$.
 Assume \eqref{elliptic}, \eqref{concave}, \eqref{addistruc}, $\lambda(A)\in \Gamma$ and
 $f\in C^\infty(\Gamma)\cap C(\overline{\Gamma})$ with $f|_{\partial \Gamma}=0$.
%Suppose in addition that there is a smoothly admissible function $\varphi$ on $X$.
Then the following conclusions are true:
\begin{itemize}
\item For each smooth function $\psi$ with $e^\psi\leq f(\lambda(A))$, there is a unique function $u$  that is smooth and admissible when restricted to $M\setminus \Sigma$, to satisfy
$f(\lambda(U[u]))=e^\psi \mbox{ in } M\setminus \Sigma$.
\item There is $u$  which is $C^{1,1}$ when restricted to $M\setminus \Sigma$, to satisfy
$f(\lambda(U[u]))=0 \mbox{ in } M\setminus \Sigma$.
\end{itemize}
In the both two cases, $u|_{\Sigma}=0$. %$(u-\varphi)|_{\Sigma}=0$.
\end{corollary}

  %The notation of admissible subsolution satisfying \eqref{existenceofsubsolution} and of strictly admissible subsolution obeying \eqref{existenceofsubsolution-de} plays important roles in the development of such equations. The advantage of this notion is that it  relaxes geometric restriction to boundary and so is applicable for some geometric problems.
%\vspace{1mm}
In order to solve the degenerate equation, we approximate it by a sequence of non-degenerate equations with
\begin{equation}
 \label{nondegenerate}
 %\delta_{\psi, f}\equiv
\delta_{\psi,f}:= \inf_{M} \psi -\sup_{\partial \Gamma} f >0.
 \end{equation}
% Such a constant $\delta_{\psi,f}$  measures whether  or not the equation is degenerate. More explicitly, if $\delta_{\psi,f}>0$ (respectively, $\delta_{\psi,f}$ vanishes) then the equation is called non-degenerate (respectively, degenerate).
%Moreover, $\sup_M\psi<\sup_{\Gamma} f$ is necessary for the solvability of equation \eqref{mainequ} within the framework of elliptic equations,
%which is automatically satisfied when $\sup_{\Gamma}f=+\infty$ or there is %an \textit{admissible} subsolution
%a subsolution satisfying \eqref{existenceofsubsolution} or \eqref{existenceofsubsolution2}.
%To prove this theorem,
Besides, %the main goal
the crucial ingredient %for Theorem \ref{mainthm1-de}
  is to establish a quantitative boundary estimate for such non-degenerate equations which says that second order
   estimate at the boundary
can be bounded from above by a %uniformly positive
constant %that is independent of
depending not on $(\delta_{\psi,f})^{-1}$.\footnote{%Throughout this paper
We say $C$ is independent of $(\delta_{\psi, f})^{-1}$ if it remains uniformly bounded   as $\delta_{\psi, f}\rightarrow 0$.}
More precisely,

 \begin{theorem}
 \label{quan-boundaryestimate-thm1}
  Let $(M,g)$ be a compact Riemannian manifold with   smooth
   \textit{mean concave} boundary,
 $\psi\in C^{1}(\bar M)$, $\varphi\in C^{3}(\partial M)$,
and we assume \eqref{elliptic}, \eqref{concave} and \eqref{nondegenerate}. %and \eqref{addistruc}.
Suppose there is an admissible subsolution $\underline{u}\in C^{2}(\bar M)$ satisfying
\begin{equation}
\label{existenceofsubsolution}
%\left\{
\begin{aligned}
f(\lambda(U[\underline{u}])) \geq  \psi
    \mbox{ in } M,  \mbox{  and  }
\underline{u}=  \varphi  \mbox{ on } \partial  M.
\end{aligned}
%\right.
\end{equation}
Then for any admissible solution $u\in C^3(M)\cap C^2(\bar M)$ to %Dirichlet problem
 \eqref{mainequ0}-\eqref{mainequ1},
there exists a uniformly positive constant depending only on $|\varphi|_{C^{3}(\bar M)}$, $|\psi|_{C^{1}(\bar M)}$,
$|\underline{u}|_{C^{2}(\bar M)}$, $\partial M$
up to its third derivatives
and other known data %(but neither on $\sup_{M}|\nabla u|$ nor on $(\delta_{\psi,f})^{-1}$).
 (but not on $(\delta_{\psi,f})^{-1}$) such that %the quantitative boundary estimate \eqref{good-quard} holds.
 \begin{equation}
 \label{good-quard}
 \begin{aligned}
 \sup_{\partial M} \Delta u \leq C %(1+\sup_{\partial M}|\nabla u|^2)
 (1+\sup_M |\nabla u|^2).
 \end{aligned}
 \end{equation}
% Moreover, if the boundary data $\varphi$ is a constant then $C$ depends  only on   $|\psi|_{C^{1}(\bar M)}$,
%$|\underline{u}|_{C^{2}(\bar M)}$, $\partial M$ up to its second derivatives and other known data.
 \end{theorem}

 The hypothesis of \textit{mean concave} boundary is only used to derive such a quantitative boundary estimate %\eqref{good-quard}.
% Besides allowing %us to derive gradient bound by using
%the  blow-up argument,
%Such a quantitative boundary estimate
that is primarily used to deal with %applies to
degenerate equations on such Riemannian manifolds
and %also imposes some new features on regularity assumptions to boundary and boundary data.
%Interestingly, our quantitative boundary estimate in Theorem \ref{quan-boundaryestimate-thm1}
%indeed depends  on boundary and boundary value up to their third derivatives, which
further weaken the regularity assumptions on the
boundary and boundary data as well. More precisely, it   only requires
  $\varphi$, $\partial M\in C^{2,1}$ in
Theorem \ref{proposition-quar-yuan1-thm}, while such regularity assumptions on  boundary and boundary data
are impossible for homogeneous real Monge-Amp\`ere equation on certain
bounded domains $\Omega\subset\mathbb{R}^n$, as the counterexamples in %Caffarelli-Nirenberg-Spruck
 \cite{CNS-deg} show that $C^{3,1}$-regularity assumptions on boundary and boundary data
are optimal for $C^{1,1}$ regularity of weak
solution to homogeneous real Monge-Amp\`ere equation on $\Omega$.
 \begin{theorem}
 \label{proposition-quar-yuan1-thm}
 Let $(M,g)$ be a compact Riemannian manifold with $C^{2,1}$-smooth
strictly \textit{mean concave} boundary of $\mathrm{H}_{\partial M}<0$. % $0<\beta<1$.
% Let $\nu$ be the unit inner normal vector along the boundary.
 Suppose, in addition to \eqref{elliptic}, \eqref{concave},  \eqref{addistruc},   % and the given data satisfies
 $\varphi\in C^{2,1}(\partial M)$, $\inf_M \psi=\sup_{\partial \Gamma}f$ and $\psi\in C^{2}(\bar M)$, that
 Dirichlet problem \eqref{mainequ0}-\eqref{mainequ1}   has
 a strictly $C^{2,1}$-smooth admissible subsolution. % with
 %$\nabla_\nu \underline{u}|_{\partial M}\leq 0 \mbox{ or }  \nabla_\nu \underline{u}|_{\partial M}\geq 0.$
%where as stated above $\nu$ is the unit inner normal vector along $\partial M$.
 %$\frac{\partial \underline{u}}{\partial \nu}|_{\partial M}\leq 0$ or $\frac{\partial \underline{u}}{\partial \nu}|_{\partial M}\geq 0$.
 Then the Dirichlet problem admits a   $C^{1,1}$ weak solution with  $\lambda(\mathfrak{g}[u])\in \overline{\Gamma}$
   and $\Delta u\in L^\infty(\bar M)$.

%Let $f\in C^{\infty}(\Gamma)\cap C(\bar \Gamma)$.
% With \eqref{nondegenerate},  $\psi\in C^{2}(\bar M)$
% and the existence of admissible subsolution replaced by $\inf_M \psi=\sup_{\partial \Gamma}f$, $\psi\in C^{1,1}(\bar M)$ and \eqref{existenceofsubsolution-de}, respectively,  we have the existence of $C^{1,1}$ weak solution to the Dirichlet problem for degenerate equations.
 \end{theorem}

The paper is organized as follows. %In Section \ref{Sketchtheproof} we sketch the proof of quantitative boundary estimate.
In Section \ref{Proofofmainestimate} we mainly derive the quantitative boundary estimate and then apply it to complete the proof of main theorem.
%In Section \ref{solvingequation} we prove the existence of equations via continuity method and approximation.
 In Section \ref{Kahlercase} we summarize the analogous result in the counterpart of complex variables.
 In Appendix \ref{appendix}, under certain assumptions, we can construct strictly
  admissible subsolutions for equations on certain topologically product spaces.
  In Appendix \ref{append-B}, with an additional assumption, we briefly discuss quantitative boundary estimate in more general case.
%In Appendix \ref{appendix-bdr}, we discuss the equations with homogeneous boundary data on more general Riemannian manifolds without %restriction to second fundamental form of boundary.

%\vspace{1mm}
%The author would like to express his gratitude to Prof. Bo Guan, Prof. Chunhui Qiu and Prof. Xi Zhang for their encouragement and support.
%%The author is  supported by NSFC grant 11801587.
%%The author would like to express his gratitude to Prof. Bo Guan, Prof. Chunhui Qiu and Prof. Xi Zhang for their encouragement and support.
%%The author is  supported by the National Natural Science Foundation of China (Grant No. 11801587).
%%The author is partially supported by NSFC (Grant No. 11801587).
%%\noindent {\bf Acknowledgements.}
% The author is supported by NSF in China, grant 11801587.
%The author is  supported in part by NSFC grant 11801587.
%The author is supported by the National Natural Science Foundation of China (Grant No. 11801587).
%The author is  supported in part by NSF in China, No. 11801587.
%NSFC (Grant No. 11801587).

% ***************************************
%\section{Preliminaries}

%\section{Quantitative boundary estimates}
\section{Proof of main results}
\label{Proofofmainestimate}

%\subsection{Quantitative boundary estimate}
%\vspace{1.5mm}

%\vspace{1mm}
%\subsubsection{Preliminaries}
%To complete the proof, we make use of the structure of equation \eqref{mainequ0}.
%\noindent  \textit{Equation \eqref{mainequ0} is of the form \eqref{mainequ20172019}}.

\subsection{Sketch of proof of Theorem \ref{quan-boundaryestimate-thm1}} \label{Sketchtheproof}

 First of all, we see that equation \eqref{mainequ0} is of the form \eqref{mainequ20172019} with
 $$\mathfrak{g}[u]=\chi+\nabla^2 u+W(d u),$$
 % $\mathfrak{g}[u]= \nabla^2 u+\frac{1}{n-1}(\mathrm{tr}_g A)g-A+du\otimes\eta+\eta\otimes du.$ For simplicity
  where $\chi=\frac{1}{n-1}(\mathrm{tr}_g A)g-A$. % and $W(du)=d u\otimes\eta+\eta\otimes d u$.
  %Thus $$\mathfrak{g}[u]=\chi+\nabla^2 u+W(d u).$$
 Then, in main equation \eqref{mainequ0}, $U[u]=(\mathrm{tr}_g \mathfrak{g}[u])g-\mathfrak{g}[u]$.

 For simplicity  we denote $\mathfrak{g}=\mathfrak{g}[u] \mbox{ and }\mathfrak{\underline{g}}=\mathfrak{g}[\underline{u}]$ for $u$
and $\underline{u}$, respectively.
One denotes
 \begin{equation}
 \label{eigenvalues-denote1}
\begin{aligned}
 \lambda(\mathfrak{g}[u])=(\lambda_1,\cdots,\lambda_n) \mbox{ and } \lambda(U[u])=(\mu_1,\cdots,\mu_n), \nonumber
 \end{aligned}
\end{equation}
 then $\mu_i=\lambda_1+ \cdots+ \hat{\lambda}_i+\cdots +\lambda_n$.
%\begin{equation}\begin{aligned}
%U[u]=(\mathrm{tr}_g \mathfrak{g}[u])g-\mathfrak{g}[u].
% \end{aligned}\end{equation}
 Let $P': \Gamma \longrightarrow P'(\Gamma)=:\widetilde{\Gamma}$ be a map given by
\begin{equation}
\begin{aligned}
(\mu_1,\cdots,\mu_n)  \longrightarrow (\lambda_1,\cdots,\lambda_n)=(\mu_1,\cdots,\mu_n) Q^{-1}, \nonumber
\end{aligned}
\end{equation}
where $Q=(q_{ij})$ and $q_{ij}=1-\delta_{ij}$ ($Q$ is symmetric). Here $Q^{-1}$ is well defined, since
%straightforward computation implies
%the eigenvalues of $Q$ are precisely $(n-1,-1,\cdots,-1)$, and
 $\mathrm{det}Q=(-1)^{n-1} (n-1)\neq 0$.
Then $\widetilde{\Gamma}$ is also an open symmetric convex cone of $\mathbb{R}^n$, and
we thus define $\tilde{f}: \widetilde{\Gamma}\rightarrow \mathbb{R}$ by $f(\mu)=\tilde{f}(\lambda).$
%\begin{equation}\begin{aligned} f(\mu)=\tilde{f}(\lambda). \nonumber \end{aligned}\end{equation}
So equation \eqref{mainequ0} is %of the form %rewritten as
\begin{equation}
\label{mainequ2}
\begin{aligned}
\tilde{f}(\lambda(\mathfrak{g}[u]))%=f(\lambda(U[u]))
=\psi. \nonumber
\end{aligned}
\end{equation}
That is, equation \eqref{mainequ0} is of the form \eqref{mainequ20172019}.
We can verify that if $f$ satisfies \eqref{elliptic}, \eqref{concave}  and \eqref{addistruc} in ${\Gamma}$, then so does $\tilde{f}$ in $\widetilde{\Gamma}$.
%See also Remark 8.11 of \cite{yuan2019}.
%Hence we can apply Proposition \ref{mix-prop1} to equation \eqref{mainequ0}
%and then complete the proof of Theorem \ref{quan-boundaryestimate-thm1}.
% Moreover, if $\lambda_1\geq \cdots \geq \lambda_n$, then
%\begin{equation}\label{special-struc1}\begin{aligned}
%\frac{\partial \tilde{f}}{\partial \lambda_i}=\sum_{j\neq i} \frac{\partial f}{\partial \mu_j}\geq  \frac{\partial f}{\partial \mu_1}\geq  \frac{1}{n(n-1)} \sum_{k=1}^n \frac{\partial \tilde{f}}{\partial \lambda_k}, \mbox{ for each } i\geq 2.
%\end{aligned}\end{equation}

The quantitative boundary estimate follows immediately from Propositions \ref{proposition-quar-yuan1} and \ref{mix-prop1}.

 %\vspace{1.5mm}
\begin{proposition}
\label{proposition-quar-yuan1}
Let $(M, g)$ be a compact Riemannian manifold with %$C^2$
\textit{mean concave} boundary. % of $\mathrm{H}_{\partial M}\leq0$.
 %whose mean curvature is nonpositive (i.e.  $\mathrm{H}_{\partial M}\leq0$).
 %Let $\nu$ be the unit inner normal vector along the boundary as defined in Theorem \ref{proposition-quar-yuan1-thm}.
Suppose, in addition to \eqref{elliptic}, \eqref{concave}, \eqref{nondegenerate},
 that for $\psi\in C^0(\bar M)$ and $\varphi\in C^2(\partial M)$
 there is a $C^2$-admissible subsolution obeying \eqref{existenceofsubsolution}.
%Let $u\in C^3(M)\cap C^{2}(\bar M)$ be an \textit{admissible} solution
% to Dirichlet problem \eqref{mainequ0}-\eqref{mainequ1}. % with the right-hand side $\psi\in C^0(\bar M)$.
Let $u\in   C^{2}(\bar M)$ be an \textit{admissible} solution
 to Dirichlet problem \eqref{mainequ0}-\eqref{mainequ1}.
Fix $x_0\in \partial M$.
Then
there is a uniformly positive constant $C$ depending  only on   $|u|_{C^0(\bar M)}$,
$|\underline{u}|_{C^{2}(\bar M)}$, $\partial M$ up to second order derivatives
 and other known data
 (but neither on $\sup_{M}|\nabla u|$ nor on $(\delta_{\psi,f})^{-1}$) such that
\begin{equation}
\label{yuan-prop-dual}
\begin{aligned}
% \mathfrak{g}(\nu,\nu)(x_0)
 \mathrm{tr}_g(\mathfrak{g})(x_0) \leq C\left(1 + \sum_{\alpha=1}^{n-1} |\mathfrak{g}(\mathrm{e}_\alpha, \nu)(x_0)|^2\right), % \nonumber
\end{aligned}
\end{equation}
where $\mathrm{e}_\alpha, \mathrm{e}_\beta \in T_{\partial M, x_0}$  $(\alpha,\beta=1,\cdots, n-1)$ with
  $g(\mathrm{e}_\alpha, \mathrm{e}_{\beta})(x_0) =\delta_{\alpha\beta}$.
\end{proposition}

%\begin{remark}
%With replaced $\mathrm{H}_{\partial M}\leq0$ by the condition of the Levi form of $\partial M$, denoted by $\mathcal{L}_{\partial M}$, has nonpositive trace (i.e. $\mathrm{tr}(\mathcal{L}_{\partial M})\leq0$), we have such a proposition in complex variables.
%\end{remark}
This proposition is a crucial ingredient for studying degenerate equation \eqref{mainequ0} and for weakening the regularity assumptions on boundary and boundary data as well.
The proposition of this type is first proved by the author \cite{yuan2017} for equation \eqref{mainequ-Kahler} (with $\eta^{1,0}=0$) on compact
Hermitian manifolds with \textit{Levi flat} boundary that is later extended
 to more general case that $\partial M$ is \textit{pseudoconcave}  in \cite{yuan2019}.
 In final section of the same paper \cite{yuan2019},
  it was  extended to   more general fully nonlinear elliptic equations than Monge-Amp\`ere equation for
 $(n-1)$-plurisubharmonic functions associated to Gauduchon's conjecture on complex manifolds with mean pseudoconcave boundary.
 %(indeed the proof works when the Levi form of boundary has nonpositive trace).
 % (as stated there, the proof works for more general cases).
 Clearly, when $\partial M$ is \textit{concave}
 the proof presented there automatically works for Dirichlet problem \eqref{mainequ20172019} and \eqref{mainequ1}
 in real variables, while for  Dirichlet problem \eqref{mainequ0}-\eqref{mainequ1}
%We observe that the boundary
%The restrictions to boundary in these propositions are %that we need to assume
% Levi form  $\mathcal{L}_{\partial M}\leq 0$ and
% $\mathrm{II}_{\partial M}\leq 0$ respectively. % (in the complex and real variable settings respectively).
% Comparing with it, we only need to assume  $\mathrm{H}_{\partial M}\leq0$ in Proposition \ref{proposition-quar-yuan1}.
%In Proposition \ref{proposition-quar-yuan1} we %need to
we assume  in Proposition \ref{proposition-quar-yuan1} that $\partial M$ is only \textit{mean concave} that is much more general than the condition of \textit{concave} for $n\geq3$. %$\mathrm{H}_{\partial M}\leq0$.
% for Dirichlet problem \eqref{mainequ0}-\eqref{mainequ1}.
 %So, Proposition \ref{proposition-quar-yuan1} can be viewed as a complement of them.
 Moreover, with replaced $\mathrm{H}_{\partial M}\leq0$ by the condition that the Levi form of $\partial M$,
 denoted by $\mathcal{L}_{\partial M}$,
  has nonpositive trace, % (i.e. $\mathrm{tr}_{\omega'}(\mathcal{L}_{\partial M})\leq0$),
  we have a similar proposition in complex variables.
 % Such propositions further enable us to obtain the existence $C^{2,\alpha}$ solution to Dirichlet problem with homogeneous boundary data
 % even if  $\partial M\in C^{2,\beta}$.
% We also refer to \cite{yuan2019Kahlercone} for follow-up work on K\"ahler cones.

\vspace{1mm}
% To complete the proof of Theorem \ref{quan-boundaryestimate-thm1} we need to prove
 % bound tangential-normal derivatives for  \eqref{mainequ20172019}
  %The proof of Proposition \ref{mix-prop1} below is almost parallel to that of Proposition 4.2 in \cite{yuan2017}
%as well as of Proposition 4.7 in \cite{yuan2019} for the equations on Hermitian manifolds
%on which the boundary is assumed to be \textit{holomorphically flat}.
%In contrast with the rigidity of complex structure of $M$ and the CR structure  on $\partial M$ as well, the differential %smooth
% structure is soft enough to %allows one to prove that
%impose without restriction to second fundamental form of boundary that the bounds %of $u$
%for mixed derivatives at the boundary can be controlled linearly by %one power of %maximum
%$L^\infty$-norm of gradient terms.
According to Proposition \ref{proposition-quar-yuan1}, the other issue for % blow-up argument %to establish \eqref{good-quard}
our %argument
goal is to prove that
 the bounds for mixed derivatives at the boundary can be controlled linearly by $L^\infty$-norm of gradient term. That is
  \begin{proposition}
 \label{mix-prop1}
 Let $(M,g)$ be a compact Riemannian manifold with boundary (but without restriction to second fundamental form of $\partial M$),
% $\nu$ be as in Theorem \ref{proposition-quar-yuan1-thm},
 and   $\psi\in C^{1}(\bar M)$, $\varphi\in C^{3}(\partial M)$.
 %Suppose the assumptions of Theorem \ref{quan-boundaryestimate-thm1} hold.
In addition to \eqref{elliptic}, \eqref{concave} and \eqref{nondegenerate}, we suppose  Dirichlet problem
 \eqref{mainequ20172019} and \eqref{mainequ1} has a $C^2$ %admissible
 subsolution with satisfying
  \begin{equation}
 \label{existenceofsubsolution20172019}
 \begin{aligned}
 f(\lambda(\mathfrak{g}[\underline{u}]))\geq \psi, \mbox{ } \lambda(\mathfrak{g}[\underline{u}])\in \Gamma \mbox{ in } M, \mbox{  }
 u|_{\partial M}=\varphi.
 \end{aligned}
 \end{equation}
 Then each solution $u\in C^3(M)\cap C^2(\bar M)$ to Dirichlet problem
 \eqref{mainequ20172019} and \eqref{mainequ1} with  $\lambda(\mathfrak{g}[u])\in \Gamma$
 satisfies
\begin{equation}
\label{quanti-mix-derivative-1}
\begin{aligned}
|\nabla^2 u(X,\nu)|\leq C %(1+\sup_{\partial M}|\nabla u|)
(1+\sup_{M}|\nabla u|), \mbox{  }\forall X\in  T_{\partial M}   \mbox{ with } |X|=1,
\end{aligned}
\end{equation}
where $C$ depends on $|\varphi|_{C^{3}(\bar M)}$, $|\psi|_{C^{1}(\bar M)}$,
$|\underline{u}|_{C^{2}(\bar M)}$, $\partial M$
up to its third derivatives
and other known data (but neither on $\sup_{M}|\nabla u|$ nor on $(\delta_{\psi,f})^{-1}$).
 \end{proposition}

\subsection{Bounds for $\sup_{M}|u|$ and $\sup_{\partial M}|\nabla u|$}
%\subsubsection{$C^0$-estimate and boundary gradient estimate}
Let $x_0\in\partial M$, $\sigma$ be the distance function to boundary, and $\rho$ be the distance function to the given point $x_0$, and $$\Omega_\delta=\{x\in M: \rho(x)<\delta\}.$$
 %We %are able to
% derive $C^0$-estimate,   gradient  estimate on boundary and the boundary estimates for double tangential derivatives
%by constructing
We construct the supersolution $w\in C^2(\bar M)$ by solving $ \mathrm{tr}_g(\mathfrak{g}[w]) = 0   \mbox{ in } M,  \mbox{  }
w  =\varphi    \mbox{ on } \partial  M.$
%\begin{equation}\label{supersolution} \begin{aligned}
% \mathrm{tr}_g(\mathfrak{g}[w]) = 0   \mbox{ in } M,  \mbox{  } w=\varphi    \mbox{ on } \partial  M.
%\end{aligned} \end{equation}
%and we assume $\sigma$ is  the distance function to $\partial M$.
The existence follows from the theory of PDEs, and the maximum principle %and comparison principle
implies $\underline{u}\leq u  \leq w   \mbox{ in } M,$
%\begin{equation} \label{daqindiguo1} \begin{aligned}
 % \underline{u}\leq u  \leq w   \mbox{ in } M,
%\end{aligned}\end{equation}
and then yields,   %there holds
on $\partial M$,
\begin{equation}
\label{formular3}
\begin{aligned}
\nabla_{\mathrm{e}_\alpha}u\,&=\nabla_{\mathrm{e}_\alpha}\underline{u}, \mbox{  }
 % \nabla_{\nu}(u-\underline{u})\geq 0, \nabla_{\nu}(w-u)\geq 0,
 \nabla_{\nu}\underline{u}\leq  \nabla_{\nu}u\leq \nabla_\nu w, \\
\nabla^2 (u-\underline{u})(\mathrm{e}_\alpha,\mathrm{e}_\beta)
\,& =-\nabla_{\nu}(u-\underline{u}) \mathrm{II}_{\partial M} (\mathrm{e}_\alpha,\mathrm{e}_\beta), \mbox{ } \forall  \mathrm{e}_\alpha, \mathrm{e}_\beta\in T_{\partial M}.
\end{aligned}
\end{equation}
% which follows from  $u=\underline{u}= w=\varphi  \mbox{ on } \partial M$,
%where $\nu$ is the unit inner normal vector along $\partial M$ as stated in Theorem \ref{proposition-quar-yuan1-thm}.
  %that  there is a uniform constant $C$ depending only on $\sup_{\bar M} |w|$,
  % $\sup_{\bar M}|\underline{u}|$, $\sup_{\partial M} |\nabla w|$ and $\sup_{\partial M} |\nabla  \underline{u}|$ such that
%\begin{equation}
%\begin{aligned}
%|u |_{C^{0}(\bar M)}+|\nabla u |_{C^{0}(\partial  M)} \leq C. \nonumber
%\end{aligned}
%\end{equation}
Thus
\begin{equation}
\label{c0-bdr-c1}
\begin{aligned}
\sup_{M}|u| + \sup_{\partial M}|\nabla u |  \leq C,
\end{aligned}
\end{equation}
%Moreover, we can write $u-\underline{u}=h\sigma \mbox{ in } \{z\in   M: \sigma< \delta\} \mbox{ for } 0<\delta\ll 1$ and  $h \in C^2$. Thus, we have
\begin{equation}
\label{puretangential}
\begin{aligned}
\sup_{\partial M}|\nabla^2 u(\mathrm{e}_\alpha,\mathrm{e}_\beta)  |\leq C,
\mbox{ } \forall  \mathrm{e}_\alpha, \mathrm{e}_\beta\in T_{\partial M}, |\mathrm{e}_\alpha|=|\mathrm{e}_\beta|=1.
\end{aligned}
\end{equation}
%here we use  %$u=\underline{u}= w=\varphi  \mbox{ on } \partial M$, and

%Notice \eqref{daqindiguo1} and the boundary value condition give $\nabla_{\nu}(u-\underline{u})\geq 0$ on $\partial M$.
%where $\nabla^2 u$ denotes the real Hessian of $u$.

%\subsubsection{Completion of the proof of Theorem \ref{quan-boundaryestimate-thm1}}
%\subsection{The structure of equation \eqref{mainequ0}}
%\subsection{Equation \eqref{mainequ0} is of the form \eqref{mainequ20172019}}

\subsection{Proof of Proposition \ref{proposition-quar-yuan1}}
To complete the proof, %of Proposition \ref{proposition-quar-yuan1},
we need the following lemma % key ingredient
which follows from an idea and  deformation argument %used in proof
from the Lemma 3.1 in \cite{yuan2017} %(see also Lemma 2.2 of \cite{yuan2019}).
(or equivalently Lemma 2.2 of \cite{yuan2019}).
%(or equivalently Lemma 2.2 of \cite{yuan2019}, see also Lemma 2.5 of \cite{yuan2019Kahlercone}). %and the deformation method used there.
%It is also a variant of  Lemma 3.1 in \cite{yuan2017}. %or equivalently  Lemma 2.2 of \cite{yuan2019}.
\begin{lemma}
 %[\cite{yuan2017}]
 \label{yuan'slemma2}
Let $\mathrm{H}$ be an $n\times n$ %Hermitian
symmetric matrix of the form
\begin{equation}\label{matrix3}\left(\begin{matrix}
\mathrm{{\bf a}}+d_1&&  &&a_{1}\\ &\mathrm{{\bf a}}+d_2&& &a_2\\&&\ddots&&\vdots \\ && &  \mathrm{{\bf a}}+d_{n-1}& a_{n-1}\\
 a_1& a_2&\cdots&  a_{n-1}& d_n\nonumber
\end{matrix}\right)\end{equation}
with $d_1,\cdots, d_{n}, a_1,\cdots, a_{n-1}$ fixed, and with $\mathrm{{\bf a}}$ variable.
Denote $\lambda_1,\cdots, \lambda_n$ by   the eigenvalues of $\mathrm{H}$.
%With the same notation in Lemma \ref{refinement3}.
Let $\epsilon>0$ be a fixed constant.
Suppose that  the parameter $\mathrm{{\bf a}}$ in $\mathrm{H}$ satisfies  the quadratic
 growth condition
  \begin{equation}
 \begin{aligned}
\label{guanjian1-yuan}
\mathrm{{\bf a}}\geq \frac{2n-3}{\epsilon}\sum_{i=1}^{n-1}|a_i|^2 +(n-1)\sum_{i=1}^{n} |d_i| + \frac{(n-2)\epsilon}{2n-3}. \nonumber
\end{aligned}
\end{equation}
 Then the eigenvalues (possibly with an order) behavior like
\begin{equation}
\begin{aligned}
 |\mathrm{{\bf a}}+d_{\alpha}-\lambda_{\alpha}|
<   \epsilon, \mbox{  }\forall 1\leq \alpha\leq n-1,
 \mbox{  }
|d_n-\lambda_{n}|
< (n-1)\epsilon. \nonumber
\end{aligned}
\end{equation}
\end{lemma}

\vspace{1mm}
%The proof follows the outline of proof of  Proposition 4.1 of \cite{yuan2017}.
%Based on Lemma \ref{yuan'slemma2},
\begin{proof}
[Proof of Proposition \ref{proposition-quar-yuan1}]
%We now give the proof of Proposition \ref{proposition-quar-yuan1} which also
%The proof follows originally from that of
%[Proposition 4.1 of \cite{yuan2017}, Proposition 8.9 from second version of \cite{yuan2019}].
Inspired  by  an idea from  \cite{yuan2019}, we %follow  the outline of  Proposition 8.9  and
 give the proof
  based on the structure of \eqref{mainequ0}.

Given $x_0\in \partial M$. %In the proof   we use the notation and local coordinate \eqref{goodcoordinate1}.
We choose  local coordinates $x=(x_1,\cdots,x_n)$ centered at $x_0$,
such that  at $x_0$, $g_{ij}=\delta_{ij}$ and $ ({\mathfrak{g}}_{\alpha\beta})$ is diagonal, $\frac{\partial}{\partial x_n}$  is
the unit inner normal vector.  In what follows all the discussions
  will be given at $x_0$, and  the Greek letters $\alpha, \beta$ range from $1$ to $n-1$.
Let's denote
\begin{equation}
{A}(R)=\left(
\begin{matrix}
R-\mathfrak{{g}}_{11}&&  &-\mathfrak{g}_{1 n}\\
%&R-\mathfrak{{g}}_{22}&& &-\mathfrak{g}_{2n}\\
&\ddots&&\vdots \\
& &  R-\mathfrak{{g}}_{{(n-1)} (n-1)}&- \mathfrak{g}_{(n-1) n}\\
-\mathfrak{g}_{n 1}&\cdots& -\mathfrak{g}_{n{(n-1)}}& \sum_{\alpha=1}^{n-1}\mathfrak{g}_{\alpha\alpha}  \nonumber
\end{matrix}
\right),
\end{equation}
\begin{equation}
\underline{A}(R)=\left(
\begin{matrix}
R-\mathfrak{{g}}_{11}&&  &-\mathfrak{g}_{1 n}\\
%&R-\mathfrak{{g}}_{22}&& &-\mathfrak{g}_{2n}\\
&\ddots&&\vdots \\
& & R-\mathfrak{{g}}_{{(n-1)} (n-1)}&- \mathfrak{g}_{(n-1) n}\\
-\mathfrak{g}_{n 1}&\cdots& -\mathfrak{g}_{n{(n-1)}}& \sum_{\alpha=1}^{n-1}\underline{\mathfrak{g}}_{\alpha\alpha}  \nonumber
\end{matrix}
\right),
\end{equation}
and
\begin{equation}
B(R)=\left(
\begin{matrix}
R-\mathfrak{\underline{g}}_{11}&-\mathfrak{\underline{g}}_{1 2}&\cdots &-\mathfrak{\underline{g}}_{1{(n-1)}} &-\mathfrak{\underline{g}}_{1n}\\
-\mathfrak{\underline{g}}_{21} &R-\mathfrak{\underline{g}}_{22}&\cdots& -\mathfrak{\underline{g}}_{2{(n-1)}}&-\mathfrak{\underline{g}}_{2n}\\
\vdots&\vdots&\ddots&\vdots&\vdots \\
-\mathfrak{\underline{g}}_{(n-1)1}&-\mathfrak{\underline{g}}_{(n-1)2}& \cdots&  R-\mathfrak{\underline{g}}_{{(n-1)} {(n-1)}}& -\mathfrak{\underline{g}}_{(n-1)n}\\
-\mathfrak{\underline{g}}_{n1}&-\mathfrak{\underline{g}}_{n2}&\cdots&-\mathfrak{\underline{g}}_{n {(n-1)}}& \sum_{\alpha=1}^{n-1}\underline{\mathfrak{g}}_{\alpha\alpha}  \nonumber
\end{matrix}\right).\end{equation}
%In particular, $A(\mathrm{tr}_g(\mathfrak{g}))=\left(U_{ij}([u])\right)$ and $B(\mathrm{tr}_g(\mathfrak{\underline{g}}))=\left(U_{ij}([\underline{u}])\right)$.
%We also denote the eigenvalues of $(n-1)\times (n-1)$ matrix $\left(\mathfrak{\underline{g}}_{\alpha\beta}\right)$ by
% $\underline{\lambda}'=(\underline{\lambda}'_1,\cdots, \underline{\lambda}'_{n-1})$.
In particular, $A(\mathrm{tr}_g(\mathfrak{g}))=\left(U_{ij}([u])\right)$ and $B(\mathrm{tr}_g(\mathfrak{\underline{g}}))=\left(U_{ij}([\underline{u}])\right)$.
  Lemma \ref{yuan'slemma2}  implies $$(R-\underline{\lambda}'_1,\cdots, R-\underline{\lambda}'_{n-1},\sum_{\alpha=1}^{n-1}\underline{\mathfrak{g}}_{\alpha\alpha})\in\Gamma \mbox{ for } R\gg1,$$
here we use %$\lambda(B(\mathrm{tr}_g(\mathfrak{\underline{g}})))=\lambda(U[\underline{u}])\in \Gamma$,
the openness of $\Gamma$.

 The ellipticity and concavity of equation, couple with Lemma 6.2 in \cite{CNS3}, %therefore
 yield that
\begin{equation}
\label{A-B}
\begin{aligned}
F(A)-F(B)\geq F^{i j}(A)(a_{i j}-b_{i j}) \nonumber
\end{aligned}
\end{equation}
for  the symmetric matrices $A=\{a_{i j} \}$ and $B=\{b_{i j}\}$ with $\lambda(A)$, $\lambda(B)\in \Gamma$.
One thus obtains that there is $R_1>0$ depending only on $\mathfrak{\underline{g}}$   such that
\begin{equation}
\label{opppp-riemann}
%\left\{
\begin{aligned}
  f(R_1-\underline{\lambda}'_1,\cdots, R_1-\underline{\lambda}'_{n-1},\sum_{\alpha=1}^{n-1}\underline{\mathfrak{g}}_{\alpha\alpha})\geq F(\lambda(B(\mathrm{tr}_g(\mathfrak{\underline{g}}))))\geq \psi,  \nonumber
%\mbox{ and }\,& (\mathfrak{\underline{g}}_{1\bar 1}-\varepsilon_{0}, \cdots, \underline{\mathfrak{g}}_{(n-1)\overline{(n-1)}}-\varepsilon_0, R_1) \in \Gamma.
\end{aligned}
%\right.
\end{equation}
$(R_1-\underline{\lambda}'_1,\cdots, R_1-\underline{\lambda}'_{n-1},\sum_{\alpha=1}^{n-1}\underline{\mathfrak{g}}_{\alpha\alpha})\in\Gamma,$ here we use %$\lambda(B(\mathrm{tr}_g(\mathfrak{\underline{g}})))=\lambda(U[\underline{u}])\in \Gamma$,  the openness of $\Gamma$
  the fact that if $A$ is diagonal then so is $F^{ij}(A)$.
%Firstly, as in proof of Proposition 4.1 of \cite{yuan2017},  we can apply Lemma  \ref{yuan'slemma2} to
Therefore, %we show that
 there exist two uniformly positive constants $\varepsilon_{0}$, $R_{0}$
 depending  on  $\mathfrak{\underline{g}}$  and $f$, such that
\begin{equation}
\label{opppp}
%\left\{
\begin{aligned}
\,& f(R_0-\underline{\lambda}'_1-\varepsilon_{0},\cdots, R_0-\underline{\lambda}'_{n-1}-\varepsilon_{0},\sum_{\alpha=1}^{n-1}\underline{\mathfrak{g}}_{\alpha\alpha}-\varepsilon_{0})\geq \psi
%\mbox{ and }\,& (\mathfrak{\underline{g}}_{1\bar 1}-\varepsilon_{0}, \cdots, \underline{\mathfrak{g}}_{(n-1)\overline{(n-1)}}-\varepsilon_0, R_1) \in \Gamma.
\end{aligned}
%\right.
\end{equation}
and
 $(R_0-\underline{\lambda}'_1-\varepsilon_{0},\cdots, R_0-\underline{\lambda}'_{n-1}-\varepsilon_{0},\sum_{\alpha=1}^{n-1}\underline{\mathfrak{g}}_{\alpha\alpha}-\varepsilon_{0})\in \Gamma$.
%We leave the proof of \eqref{opppp} at the end of the proof of this proposition.
Here  \eqref{elliptic}, \eqref{concave} % (or more general the level sets of $f$ in $\Gamma$ are convex),
and the openness of
$\Gamma$ are  needed. %The details of the proof can be found in \cite{yuan2017}.

Next, we apply Lemma  \ref{yuan'slemma2} to matrix $\underline{A}(R)$. Let
%\vspace{1mm}
%%Next,   Lemma  \ref{yuan'slemma2}  together with \eqref{opppp} help us to
%% establish the quantitative boundary estimates for double normal derivative.
%Let's pick  $\epsilon=\frac{\varepsilon_0}{2(n-1)}$ in  Lemma  \ref{yuan'slemma2} and set
\begin{equation}
\begin{aligned}
 R_c=\,& \frac{2(n-1)(2n-3)}{\varepsilon_0}
\sum_{\alpha=1}^{n-1} | \mathfrak{g}_{\alpha n}|^2
+(n-1)\sum_{\alpha=1}^{n-1} ( | \mathfrak{{g}}_{\alpha \alpha}| + | \mathfrak{\underline{g}}_{\alpha \alpha}|)
 +\sum_{\alpha=1}^{n-1}|\underline{\lambda}'_\alpha|
  +R_0   +\varepsilon_0, \nonumber
\end{aligned}
\end{equation}
%\begin{equation}\begin{aligned}
% R_c=\,& \frac{2(n-1)(2n-3)}{\varepsilon_0} \sum_{\alpha=1}^{n-1} | \mathfrak{g}_{\alpha n}|^2 +(n-1)\sum_{\alpha=1}^{n-1}  | \mathfrak{{g}}_{\alpha \alpha}| + \sum_{\alpha=1}^{n-1}  | \mathfrak{\underline{g}}_{\alpha \alpha}| +\sum_{\alpha=1}^{n-1}|\underline{\lambda}'_\alpha|
%    +\varepsilon_0 +R_0, \nonumber
%\end{aligned}\end{equation}
where $\varepsilon_0$ and $R_0$ are fixed constants so that \eqref{opppp} holds.
It follows from  Lemma   \ref{yuan'slemma2}    that the eigenvalues of $\underline{A}(R_c)$ (possibly with an order) shall behavior like
%\begin{equation}
%\label{lemma12-yuan}
%\begin{aligned}
%\lambda(\underline{A}(R_c)) \in
%(R_c- \mathfrak{g}_{11}-\frac{\varepsilon_0}{2(n-1)},\cdots,
%R_c- \mathfrak{g}_{(n-1)  {(n-1)}}-\frac{\varepsilon_0}{2(n-1)},\sum_{\alpha=1}^{n-1}\mathfrak{\underline{g}}_{\alpha \alpha}-\frac{\varepsilon_0}{2}) +\overline{ \Gamma}_n \subset \Gamma.
%\end{aligned}
%\end{equation}
\begin{equation}
\label{lemma12-yuan}
\begin{aligned}
\lambda(\underline{A}(R_c)) \in
(R_c- \mathfrak{g}_{11}-\frac{\varepsilon_0}{2},\cdots,
R_c- \mathfrak{g}_{(n-1)  {(n-1)}}-\frac{\varepsilon_0}{2},\sum_{\alpha=1}^{n-1}\mathfrak{\underline{g}}_{\alpha \alpha}-\frac{\varepsilon_0}{2}) +\overline{ \Gamma}_n \subset \Gamma.
\end{aligned}
\end{equation}

It follows from \eqref{formular3}, the technical hypothesis   $\sum_{i\neq k}\eta^{k}_{ii}(x_0)=0$, and  $\mathrm{H}_{\partial  M}\leq 0$ that % and the boundary value condition that
$ \sum_{\alpha=1}^{n-1}\mathfrak{g}_{\alpha\alpha} \geq \sum_{\alpha=1}^{n-1}\underline{\mathfrak{g}}_{\alpha\alpha}$.
Thus
\begin{equation}
\label{puretangential2}
\begin{aligned}
A(R)\geq \underline{A}(R).
\end{aligned}
\end{equation}
This is the only place where we use the %assumption of $\mathrm{H}_{\partial  M}\leq 0$.
mean concavity of boundary and technical hypothesis on $\eta$ as well.
Thus
  \begin{equation}
 \label{nornor}
\begin{aligned}
\mathrm{tr}_g(\mathfrak{g}) < R_c\leq C  (1 + \sum_{\alpha=1}^{n-1} |\mathfrak{g}_{\alpha n}|^2 ), \nonumber
\end{aligned}
\end{equation}
since   \eqref{elliptic},  \eqref{opppp}, \eqref{lemma12-yuan} and  \eqref{puretangential2} imply
\begin{equation}
\begin{aligned}
F(A(R_c))\geq \,& F(\underline{A}(R_c))  \\
%\geq\,&
%f(R_c- \mathfrak{g}_{11}-\frac{\varepsilon_0}{2(n-1)},\cdots,
%R_c- \mathfrak{g}_{(n-1)  {(n-1)}}-\frac{\varepsilon_0}{2(n-1)},\sum_{\alpha=1}^{n-1}\mathfrak{\underline{g}}_{\alpha \alpha}-\frac{\varepsilon_0}{2})
%\\
\geq\,&
f(R_c- \mathfrak{g}_{11}-\frac{\varepsilon_0}{2},\cdots,
R_c- \mathfrak{g}_{(n-1)  {(n-1)}}-\frac{\varepsilon_0}{2},\sum_{\alpha=1}^{n-1}\mathfrak{\underline{g}}_{\alpha \alpha}-\frac{\varepsilon_0}{2})
\\
> \,&
f(R_0-\underline{\lambda}'_1-\varepsilon_{0},\cdots, R_0-\underline{\lambda}'_{n-1}-\varepsilon_{0},\sum_{\alpha=1}^{n-1}\underline{\mathfrak{g}}_{\alpha\alpha}-\varepsilon_{0})\geq \psi. \nonumber
\end{aligned}
\end{equation}
Here we also use \eqref{formular3}.

%This completes the proof. %of  Proposition  \ref{proposition-quar-yuan1}.
%We can also apply Lemma \ref{refinement111} to  prove Theorem \ref{boundaryestimatethm}.
\end{proof}

\begin{remark}
As pointed out in \cite{yuan2017}, \eqref{opppp} only requires the existence of (admissible)
$\mathcal{C}$-subsolution. Thus, in Proposition \ref{proposition-quar-yuan1} subsolutions can be replaced by
$\mathcal{C}$-subsolution\renewcommand{\thefootnote}{\fnsymbol{footnote}}\footnote{The admissible subsolution satisfying \eqref{existenceofsubsolution} is clearly a $\mathcal{C}$-subsolution  for equation \eqref{mainequ0}.}
 with the same boundary data. %(introduced by Sz\'ekelyhidi)
Moreover, the proof is still works for
\begin{equation}\label{mainequation-general}\begin{aligned}
f(\lambda(U[u]))=\psi(x,u,\nabla u)
\end{aligned}\end{equation} (without assumptions on  precise dependences of
$u$ and $\nabla u$ on $\psi$), if we further assume %$f=(\sigma_k)^{1/k}$ or more general
\begin{equation}\label{unbound-strong} \begin{aligned}
\lim_{t\rightarrow+\infty}f(\lambda_1+t,\cdots,\lambda_{n-1}+t,\lambda_n)=\sup_\Gamma f, \mbox{  } \forall \lambda=(\lambda_1,\cdots,\lambda_n)\in \Gamma,
\end{aligned}\end{equation}
including $f(\lambda)=\sum_{i=1}^n\log\lambda_i$ corresponding to Monge-Amp\`ere type equation.
Therefore, \eqref{yuan-prop-dual} holds for equation \eqref{mainequation-general} with \eqref{mainequ1}
when $f$ satisfies \eqref{unbound-strong} and the right-hand side depends on $u$ and $\nabla u$ as well.
\end{remark}

%\begin{remark} \eqref{puretangential2} still holds, if $-H+\sum_{\alpha=1}^{n-1}\eta(d\nu,\frac{\partial}{\partial x_\alpha}, \frac{\partial}{\partial x_\alpha})\geq 0$
%\end{remark}

%\subsection{Bounds for tangential-normal derivatives}
\subsection{Proof of Proposition \ref{mix-prop1}}

%The proof of Proposition \ref{mix-prop1} is almost parallel to that of Proposition 4.2 in \cite{yuan2017}
%as well as of Proposition 4.7 in \cite{yuan2019}.
% for equations on Hermitian manifolds with \textit{holomorphically flat} boundary.
%of which the boundary is assumed to be \textit{holomorphically flat}.
%In contrast with the rigidity of complex structure of $M$ and the CR structure  on $\partial M$ as well, the differential %smooth
% structure is soft enough to %allows one to prove that impose that the bounds %of $u$
%for mixed derivatives at the boundary can be controlled linearly by %one power of %maximum
%$L^\infty$-norm of gradient terms.
%Comparing to complex variables,
The proof of Proposition \ref{mix-prop1} is almost parallel to that of Proposition 4.2 in \cite{yuan2017}
as well as of Propositions 4.6-4.8 in \cite{yuan2019} for the equations on Hermitian manifolds.
%on which the boundary is assumed to be \textit{holomorphically flat}.
%%\renewcommand{\thefootnote}{\fnsymbol{footnote}}\footnote{In the sense that there are holomorphic coordinates $(z_1, \cdots,z_n)$
%%such that $\partial M$ is locally given by $\mathfrak{Re} (z_n) = 0$.}.
%In contrast with the rigidity of complex structure of $M$ and the CR structure  on $\partial M$ as well, the differential %smooth
% structure is soft enough to allow one to flatten boundary and
%impose \eqref{quanti-mix-derivative-1} without any restriction to the second fundamental form of boundary.
%%that the bounds %of $u$ for mixed derivatives at the boundary can be controlled linearly by %one power of %maximum
%$L^\infty$-norm of gradient term.
For completeness we present the proof here.
Given $x_0\in \partial M$
 one has local coordinates
 \begin{equation}
 \label{coordinate1}
 \begin{aligned}
 x=(x_1,\cdots,x_n)
 \end{aligned}
 \end{equation}
%centered at the given point $x_0$ $(x_0=\{x=0\})$ such that %the boundary
 with origin at $x_0$ %$(x_0=\{x=0\})$
 such that $\partial M$ is locally given by $x_n=0$.
  Moreover, we can assume %$g_{\alpha\beta}(x_0)  =\delta_{\alpha\beta}$, $1\leq \alpha,\beta<n$, $g_{nn}(x_0)=1$.
  $g_{ij}(x_0)=\delta_{ij}$.
  %(Also, we could choose local frames $(e_1,\cdots,e_n)$ around $x_0$ such that, when restricted to $\partial M$, $e_n=\nu$).  %is normal to $\partial M$).

   We will carry out the computations in such local coordinates, and we set $\nabla_i=\nabla_{\frac{\partial}{\partial x_i}}$,
 $\nabla_{ij}= %\frac{\partial^2}{\partial x_i\partial x_j}-\nabla_{\frac{\partial}{\partial x_i}}\frac{\partial}{\partial x_j}=
 \frac{\partial^2}{\partial x_i\partial x_j} -\Gamma_{ji}^k \frac{\partial}{\partial x_k}$, with a similar convention for higher derivatives,
 where $\Gamma^k_{ij}$ are the Christoffel symbols defined by $\nabla_{\frac{\partial}{\partial x_j}}\frac{\partial}{\partial x_i}=\Gamma_{ji}^k \frac{\partial}{\partial x_k}$. Under Levi-Civita connection, $\Gamma_{ij}^l=\Gamma_{ji}^l$.
 %For convenience we outline the proof as in the following.
% For convenience  the Greek letters $\alpha, \beta$ range from $1$ to $n-1$.
 By direct computation one has
%\begin{equation}
%\label{direct-compute1}
%\begin{aligned}
%&\nabla_{ij}\nabla_k u=  \frac{\partial^3 u}{\partial x_i\partial x_j \partial x_k}-\Gamma_{ij}^l\frac{\partial^2 u}{\partial x_k\partial x_l}, \\
%\nabla_{ijk}u=\frac{\partial^3 u}{\partial x_i\partial x_j \partial x_k}
%\,&-\frac{\partial \Gamma_{ij}^l}{\partial x_k}\nabla_l u
%-\Gamma_{ij}^l\Gamma_{kl}^p\nabla_p u
%%\\ \,&
%-\Gamma_{ij}^l \nabla_{kl} u-\Gamma_{ki}^l\nabla_{lj}u -\Gamma_{kj}^l\nabla_{il}u.
%\end{aligned}
%\end{equation}
\begin{equation}
\label{direct-compute1}
\begin{aligned}
\,&\nabla_{ij}\nabla_k u= \frac{\partial^3 u}{\partial x_i\partial x_j \partial x_k}-\Gamma_{ij}^l\frac{\partial^2 u}{\partial x_k\partial x_l},
\end{aligned}
\end{equation}
\begin{equation}
\label{direct-compute1*}
\begin{aligned}
\nabla_{ijk}u
%=\,&
%\frac{\partial^3 u}{\partial x_i\partial x_j \partial x_k}
%-\Gamma_{ij}^l \frac{\partial^2 u}{\partial x_k \partial x_l}
%-\frac{\partial \Gamma_{ij}^l}{\partial x_k}\frac{\partial u}{\partial x_l}
%-\Gamma_{ki}^l \frac{\partial^2 u}{\partial x_j \partial x_l}
%-\Gamma_{kj}^l \frac{\partial^2 u}{\partial x_i \partial x_l} \\ \,&
%+\Gamma_{ki}^l\Gamma_{lj}^p\frac{\partial u}{\partial x_p}
%+ \Gamma_{kj}^l \Gamma_{li}^p \frac{\partial u}{\partial x_p}\\
=\,& \frac{\partial^3 u}{\partial x_i\partial x_j \partial x_k}-\Gamma_{ij}^l\frac{\partial^2 u}{\partial x_k\partial x_l}-\frac{\partial \Gamma_{ij}^l}{\partial x_k}\frac{\partial u}{\partial x_l}
+\Gamma_{ki}^l\nabla_{lj}u+\Gamma_{kj}^l \nabla_{li}u.
%=\,&
%\frac{\partial^3 u}{\partial x_i\partial x_j \partial x_k}
% -\frac{\partial \Gamma_{ij}^l}{\partial x_k}\nabla_l u
%-\Gamma_{ij}^l\Gamma_{kl}^p\nabla_p u
%-\Gamma_{ij}^l \nabla_{kl} u-\Gamma_{ki}^l\nabla_{lj}u -\Gamma_{kj}^l\nabla_{il}u.
\end{aligned}
\end{equation}
Then $\nabla_{ij}\nabla_k u=\nabla_{ijk}u+\frac{\partial \Gamma_{ij}^l}{\partial x_k}\frac{\partial u}{\partial x_l}
-\Gamma_{ki}^l\nabla_{lj}u-\Gamma_{kj}^l \nabla_{li}u $ (note that the terms $\frac{\partial \Gamma_{ij}^l}{\partial x_k}$ appear in the formula). % while it is different from that of complex variables). %(note that $\frac{\partial \Gamma_{ij}^l}{\partial x_k}$ appears in the   formula which is different from that in complex variables).
Under local coordinates \eqref{coordinate1}, %(one assumes $g_{\alpha\beta}(x_0)  =\delta_{\alpha\beta}$, $1\leq \alpha,\beta<n$),
we take the tangential operator on  boundary as
  \begin{equation}
  \label{tangential-operator12321-meng}
  \mathcal{T} %=\pm \frac{\partial}{\partial y_n},
  %\pm \nabla_{\frac{\partial}{\partial x_{\alpha}}}
  =\pm{\frac{\partial}{\partial x_{\alpha}}},
%\pm \frac{\partial}{\partial y_{\alpha}},
\mbox{ for } 1\leq \alpha \leq n-1.
\end{equation}
%}}
%(Notice $\mathcal{T}$ is just defined locally).

Let $\mathcal{L}$ be the linearized operator at $u$  of
equation \eqref{mainequ20172019}  which   is given by
\begin{equation}
\label{linearoperator2}
\begin{aligned}
\mathcal{L}v= F^{ij}\nabla_{ij} v+F^{ij}\eta_{ij}^k \nabla_k v,
 \mbox{ for } v\in C^{2}( M),  \nonumber
\end{aligned}
\end{equation}
where $F^{ij}=\frac{\partial F}{\partial a_{ij}}(\mathfrak{g}[u])$.

First of all, we have the following lemma. %(We omit the proof, as the proof is standard and straightforward).
\begin{lemma}
\label{DN}
Let $u\in C^3(M)\cap C^1(\bar M)$ be an admissible solution to equation \eqref{mainequ20172019}.
There is a positive constant $C$ depending only on
$|\varphi|_{C^{3}(\bar M)}$, $|\chi|_{C^{1}(\bar M)}$, $\psi_{C^{1}(\bar M)}$, $\partial M$
up to its third derivatives
 and other known data (but not on $(\delta_{\psi,f})^{-1}$) such that
\begin{equation}
\label{pili1}
\begin{aligned}
\left|\mathcal{L}(\mathcal{T}(u-\varphi))\right|\leq C\left(1+(1+\sup_{ M}|\nabla u|)
\sum_{i=1}^n f_i+  \sum_{i=1}^n f_i|\lambda_{i}|\right), \mbox{ in } \Omega_{\delta} % \nonumber
\end{aligned}
\end{equation}
  for some small $\delta>0$.
%Moreover, if the boundary data $\varphi$ is a constant, then $C$ depends only on $|\chi|_{C^{1}(\bar M)}$, $|\psi|_{C^{1}(\bar M)}$
% and other known data.
\end{lemma}

\begin{proof}
Differentiating equation \eqref{mainequ20172019} one has
\begin{equation}
\begin{aligned}
F^{ij}(\nabla_{ij k}u+\chi_{ij;k}+{\eta_{ij}^l}_{;k} \nabla_l u+\eta_{ij}^l\nabla_{lk}u)=\nabla_k \psi. \nonumber
\end{aligned}
\end{equation}
%where %$\eta_{i;\alpha}$ denotes the covariant derivative of $\eta_i$ along $\frac{\partial}{\partial x_\alpha}$.
%$\eta_{i;k}=\frac{\partial \eta_i}{\partial x_k}-\Gamma^l_{ik}\eta_l$.
Combining with \eqref{direct-compute1} and \eqref{direct-compute1*} one derives \eqref{pili1}.
\end{proof}

\noindent
\textit{Construction of barriers and completion of proof of Proposition \ref{mix-prop1}}.
 The proposition can be proved by constructing barrier functions similar to that used in \cite{yuan2017,yuan2019,yuan2019Kahlercone} in complex variables, and
 the construction of this type of  barriers %follows originally from
goes back at least to \cite{Hoffman1992Boundary,Guan1993Boundary,Guan1998The}.
Let's take
\begin{equation}
\label{ggg}
\begin{aligned}
\widetilde{\Psi}=
A_{1} \,&\sqrt{b_{1}}(\underline{u}-u)-A_{2}\sqrt{b_{1}}\rho^{2}+A_{3}\sqrt{b_{1}}(N\sigma^{2}-t\sigma)
%\\ \,&
+\frac{1}{\sqrt{b_1}}\sum_{\tau<n}|\nabla_\tau(u-\varphi)|^2+ \mathcal{T}(u-\varphi), \nonumber
\end{aligned}
\end{equation}
where $b_{1}=1+\sup_{  M} |\nabla (u-\varphi)|^{2}$.
%In particular, if $\varphi$ is a constant function, then
%\begin{equation}
%\begin{aligned}
%\widetilde{\Psi}=
%A_{1} \,&\sqrt{b_{1}}(\underline{u}-u)-A_{2}\sqrt{b_{1}}\rho^{2}+A_{3}\sqrt{b_{1}}(N\sigma^{2}-t\sigma)
%+\frac{1}{\sqrt{b_1}}\sum_{\tau<n}|\nabla_\tau u|^2+ \mathcal{T}u. \nonumber
%\end{aligned}
%\end{equation}

%In what follows we denote by  $\widetilde{u}=u-\varphi$.
% We shall point out that the constants appearing in the proof of quantitative boundary estimates,
%such as $C$, $C_0$, $C_1$,   $C_1'$,  $C_2$,  $ A_1, A_2, A_3, \cdots$, etc
% %do not depend on $|\nabla u|$.
%depend   neither on $|\nabla u|$ nor on $(\delta_{\psi,f})^{-1}$.
%We shall point out that the constants  %appearing
%below, such as $C$, $C_0$, $C_1$,  $C_1'$,  $C_2$, do not depend on $|\nabla u|$.

Let $\delta>0$ and $t>0$ be sufficiently small such that $N\delta-t\leq 0$
 (where $N$ is a positive constant sufficiently large to be determined later), $\sigma$ is $C^2$ %in $\Omega_{\delta}$ and
 and
\begin{equation}
\label{bdy1}
\begin{aligned}
 \frac{1}{2} \leq |\nabla \sigma|\leq 2,  \mbox{  }
  |\mathcal{L}\sigma | \leq   C_2\sum_{i=1}^n f_i,     \mbox{  }
  |\mathcal{L}\rho^2| \leq C_2\sum_{i=1}^n f_{i}, \mbox{ in } \Omega_{\delta}.
\end{aligned}
\end{equation}
Furthermore, we  can choose $\delta$ and $t$  small enough  such that $|2N\delta-t|$ and $t$ are both small.
%\begin{equation}\label{yuanbd-11}\begin{aligned}
%\max\{|2N\delta-t|, t\}\leq \min \{\frac{\varepsilon}{2C_{2}}, \frac{\beta}{16\sqrt{n}C_2} \},
%\end{aligned}\end{equation}
%where $\beta=\frac{1}{2}\min_{\bar M} dist(\nu_{\lambda[\underline{u}]}, \partial \Gamma_{n})$,
%$\varepsilon$ is the constant  in Lemma \ref{guan2014} and $C_2$ is the constant in \eqref{bdy1}.

By straightforward calculation and %an elementary inequality
$|a-b|^2\geq  \frac{1}{2}|a|^2- |b|^2$, one has
\begin{equation}
\label{bdygood1}
\begin{aligned}
\mathcal{L} (\sum_{\tau<n}|\nabla_{\tau} (u-\varphi)|^2  )
\geq
\,&
 \sum_{\tau<n}F^{ij} \mathfrak{g}_{\tau i} \mathfrak{g}_{\tau  j}
   -C_1'\sqrt{b_1} \sum_{i=1}^n  f_{i}|\lambda_{i}| %\\\,&
   -C_1' b_1\sum_{i=1}^n  f_{i} -C_1' \sqrt{b_{1}}.
\end{aligned}
\end{equation}
%where we use the elementary inequality $|a-b|^2\geq  \frac{1}{2}|a|^2- |b|^2$.

By Proposition 2.19 in  \cite{Guan12a}, there is an index $r$ such that
\begin{equation}
\label{beeee}
\begin{aligned}
\sum_{\tau<n} F^{ij}\mathfrak{g}_{\tau i}\mathfrak{g}_{\tau j}\geq
  \frac{1}{2}\sum_{i\neq r} f_{i}\lambda_{i}^{2}. \nonumber
\end{aligned}
\end{equation}
By \eqref{bdy1}, \eqref{bdygood1} and Lemma \ref{DN}, we therefore arrive at the following key inequality
\begin{equation}
\label{bdycrucial}
\begin{aligned}
\mathcal{L}(\widetilde{\Psi}) \geq \,& A_1 \sqrt{b_1} \mathcal{L}(\underline{u}-u)
+\frac{1}{2\sqrt{b_1}} \sum_{i\neq r} f_{i}\lambda_{i}^{2}
+A_3 \sqrt{b_1}\mathcal{L}(N\sigma^{2}-t\sigma)
\\
\,&
 -C_1
   -C_1 \sum_{i=1}^n  f_{i}|\lambda_{i}|
-\left(A_2C_2 +C_1 \right) \sqrt{b_1} \sum_{i=1}^n  f_{i}. %\nonumber
\end{aligned}
\end{equation}

%Similar to the proof of %Proposition 4.2 in \cite{yuan2017} as well as
Same as the discussion presented in proof of Proposition 4.7 of \cite{yuan2019} and Proposition 4.2 in \cite{yuan2017} as well, we can prove
\begin{equation}
\label{bdy001good}
\begin{aligned}
\mathcal{L}\widetilde{\Psi}
 \geq 0,   \mbox{ in } \Omega_{\delta}  \nonumber
\end{aligned}
\end{equation}
for $0<\delta\ll1$, if we appropriately choose $A_1\gg A_2\gg A_3>1$, $N\gg1$ and $0<t\ll1$.

The boundary is locally given by $x_n=0$ in local coordinates \eqref{coordinate1},
we know that   $\mathcal{T}(u-\varphi)=0$ and $\nabla_\tau(u-\varphi)=0$  on $\partial   M\cap \overline{\Omega}_{\delta}$,
 $\forall 1\leq \tau<n$. % for $\delta>0$. %where $\mathcal{T}$ is  defined in  \eqref{tangential-operator12321-meng}.
Thus %we can assume $0<t, \delta\ll 1$ and $N\delta-t\leq 0$ such that
\begin{equation}
\label{yuan-rr1}
\begin{aligned}
\widetilde{\Psi}= \,&A_{1}\sqrt{b_{1}}(\underline{u}-u)-A_{2}\sqrt{b_{1}}\rho^{2}+A_{3}\sqrt{b_{1}}(N\sigma^{2}-t\sigma)
\\ \,&
+\frac{1}{\sqrt{b_1}}\sum_{\tau<n}|\nabla_\tau(u-\varphi)|^2+ \mathcal{T}(u-\varphi)
\leq  0, \mbox{ on } \partial   M\cap  \overline{\Omega}_{\delta}.  \nonumber
\end{aligned}
\end{equation}
On the other hand, $\rho=\delta$ and $\underline{u}-u\leq 0$ on $M\cap \partial\Omega_{\delta}$.
 Hence, if  $A_2\gg 1$ then
 $\widetilde{\Psi}\leq 0$
   on  $M\cap \partial\Omega_{\delta}$,
 where we  use again $N\delta-t\leq 0$.
Therefore $\widetilde{\Psi}\leq 0$ in $\Omega_{\delta}$ by applying maximum principle.
Together with $\widetilde{\Psi}(0)=0$, one has
$\nabla_\nu\widetilde{\Psi} (0)\leq 0$.
%{\color{red}{
Thus
\begin{equation}
\begin{aligned}
  \nabla_{\nu} \mathcal{T}(u-\varphi)(0) \leq \,& - A_1\sqrt{b_1}\nabla_\nu(\underline{u}-u)(0) +A_2 \sqrt{b_1}\nabla_\nu(\rho^2)(0) \\\,&
  - A_3\sqrt{b_1}\nabla_\nu(N\sigma^2-t\sigma)(0)
  \\ \,&
  -\frac{2}{\sqrt{b_1}}\sum_{\tau<n}  (\nabla_\nu(u-\varphi)\nabla_\nu(\nabla_{\tau}(u-\varphi)))(0)     \\
 \leq \,& C(1+\sup_{\partial M}|\nabla (u-\underline{u})|)(1+\sup_M |\nabla u|)
 \leq C' (1+\sup_M |\nabla u|).  \nonumber
\end{aligned}
\end{equation}
Here we use \eqref{c0-bdr-c1}.
%The above discussion also works if we take the operator as $-\mathcal{T}$.
%}}
Therefore
%we then obtain
%Applying maximum principle Hopf to $\widetilde{\Psi}$ on $\Omega_{\delta}$($0<\delta\ll 1$), we obtain
\begin{equation}
\label{basa}
\begin{aligned}
\nabla_{\mathcal{T}\nu}u %=\pm\nabla_{\frac{\partial}{\partial x_\alpha}\nu}
=\pm\nabla_{\alpha n}u\leq C (1+\sup_{M}|\nabla u|), \mbox{ at } x_0,
\end{aligned}
\end{equation}
where $C$  depends only on  %$|u|_{C^0(\bar M)}$,
 $|\varphi|_{C^{3}(M)}$,
$|\underline{u}|_{C^{2}(M)}$
 $|\psi|_{C^{1}(M)}$, $\partial M$
up to its third derivatives
 and other known data
 (but not on $\sup_{M}|\nabla u|$).
Moreover,   the constant $C$ in \eqref{basa}  does not
 depend on $(\delta_{\psi,f})^{-1}$.

%\begin{remark} The existence $C^{2,\alpha}$ solution to %of Theorem \ref{proposition-quar-yuan1-thm} holds for
% equation \eqref{mainequ20172019} with homogeneous boundary data holds when $(M,g)$ is a compact Riemannian manifold with $C^{2,\beta}$ \textit{concave} boundary. \end{remark}

%*****************************
%{1.5mm}
%\noindent {\bf Complete the proof of Theorem \ref{quan-boundaryestimate-thm1}}.

%\begin{remark}
%To assure  $\varphi\in C^3(\partial M)$ and the computations  %of $\mathcal{L}(\widetilde{\Psi})$
%(near the boundary)
%in the proof of  Proposition \ref{mix-prop1} (and so of Theorem \ref{quan-boundaryestimate-thm1})
% make sense, the regularity of  boundary is  assumed to be of class $C^3$; while it is subtle if  boundary data is  constant.
% \end{remark}

\subsection{Completion of the proof of Theorem \ref{mainthm1-de}}

In analogy with the Theorem 4.9 proved in \cite{yuan2019}, %we can prove
if there is an admissible subsolution $\underline{u}\in C^2(\bar M)$, then
for any admissible  solution $u\in C^4(M)\cap C^2(\bar M)$ to Dirichlet problem \eqref{mainequ20172019} and \eqref{mainequ1},
there  is a uniformly positive constant $C$ depending
only on $|u|_{C^{0}(\bar M)}$,
$|\underline{u}|_{C^{2}(\bar M)}$, $|\chi|_{C^{2}(\bar M)}$,
$|\psi|_{C^2(\bar M)}$ and other known data (but not on $(\delta_{\psi,f})^{-1}$) such that
\begin{equation}
\label{2-seglobal}
\begin{aligned}
\sup_{ M}\Delta u
\leq
C (1+ \sup_{M}|\nabla u|^{2} +\sup_{\partial   M}|\Delta u|).
\end{aligned}
\end{equation}
Comparing to % \eqref{2-seglobal}
that of equation \eqref{mainequ-Kahler} involving gradient terms %that considered
in complex variables studied in \cite{yuan2019},
  the proof of %global second order estimate
 \eqref{2-seglobal}  %for equation \eqref{mainequ20172019}
  is much more simple
  in %the setting of
  real variables. We hence only summarize it but omit the details. % to cut the paper's length.
  The estimation \eqref{2-seglobal} for $\eta=0$ in real variables was mentioned in Section 8 of \cite{Gabor} where Sz\'ekelyhidi assumes the existence of
$\mathcal{C}$-subsolution, %\renewcommand{\thefootnote}{\fnsymbol{footnote}}\footnote{The admissible subsolution satisfying \eqref{existenceofsubsolution20172019} is clearly a $\mathcal{C}$-subsolution for equation \eqref{mainequ20172019}.} %for  equations, %satisfying \eqref{elliptic}, \eqref{concave}, \eqref{addistruc} and \eqref{nondegenerate},
and was also proved in  %version 2 of
the Theorem 1.6 of %current version of
 \cite{Guan14} for Dirichlet problem for more general equations with removing \eqref{addistruc}.
  With \eqref{good-quard}, %Theorem \ref{quan-boundaryestimate-thm1},
   \eqref{2-seglobal} and  $\mathrm{tr}_g (U[u])>0$ at hand, we establish
   \begin{equation}
\label{mainestimate2}
\begin{aligned}
\sup_M\Delta u\leq C(1+\sup_M|\nabla u|^2)  %\nonumber
\end{aligned}
\end{equation}
and then apply it to prove gradient bound by using a blow-up argument proposed by \cite{Chen} and \cite{HouMaWu2010,Dinew2017Kolo} for
complex Monge-Amp\`ere equation and complex $k$-Hessian equation respectively, and further by
Sz\'{e}kelyhidi \cite{Gabor} for general complex fully nonlinear elliptic equations.    % with a different method.
 %We also refer to Section 8 of \cite{Gabor} for brief discussion on  Liouville type theorem and the blow-up argument in Riemannian case.
We also refer the reader
to Section 8 of \cite{Gabor} for brief discussion on  Liouville type theorem and blow-up argument in real case.
%(Here we also use  $\mathrm{tr}_g (U[u])>0$).
Evans-Krylov theorem  \cite{Evans82,Krylov83} and Schauder theory give $C^{2,\alpha}$ and higher order regularity.
%(see also \cite{Silvestre2014Sirakov} for $C^{2,\alpha}$ regularity near the boundary).

   Indeed, % building on the work of \cite{Guan12a},
  Guan \cite{Guan12a}  and Guan-Jiao  \cite{Guan2015Jiao} %(with gradient terms)
   proved a second
estimate  %$\sup_{M}\Delta u\leq C(1+\sup_{\partial M}|\Delta u|)$
for a class of fully nonlinear elliptic
equations  possibly with gradient terms
on Riemannian manifolds, however,
the bound does not tell one how it precisely depends on the gradient bound which is not enough to apply blow-up argument.
Our quantitative boundary estimate enables us to %overcome the difficulty  and
to achieve this goal.
%so that we can apply blow-up argument to deduce the gradient bound.
%We also refer for example to \cite{Urbas2002,LiYY1990} and recent work\cite{Guan12a,Guan14} for direct proof of gradient estimate without using second estimates.
Also, for the proof of gradient estimate
 for Hessian type equations ($\eta=0$)
 on curved Riemannian manifolds,  without using second  estimate,  please refer for instance to %some literature cited in introduction among others
\cite{LiYY1990,Urbas2002,Guan12a} and %significant
progress made by Guan  (see Theorem 1.6 in current version of  \cite{Guan14}).

 It would be worthwhile to note that, besides with giving this  interesting and different approach to gradient estimate,
 %what in the role our boundary estimate play
what the major role of our quantitative boundary estimate  \eqref{good-quard}  plays %in this paper
is allowing one to deal with degenerate equations and to impose some regularity assumptions on boundary and boundary data as well,
% (see the final section of \cite{yuan2017} for $\eta=0$ and Remark 7.5 of \cite{yuan2019}).
since \textit{a priori} estimates up to second order among others boundary estimate for second derivatives are all
 independent of $(\delta_{\psi,f})^{-1}$. We thus apply approximation %and continuity method
  to study degenerate equations
  and finally derive Theorem \ref{mainthm1-de} as a complement of the following theorem.
   (See the final section of \cite{yuan2017} for $\eta=0$ and Remark 7.5 in first version of \cite{yuan2019} for the announcement on general $\eta$).

%Finally, applying continuity method and
% approximation, %together with  Evans-Krylov theorem  \cite{Evans82,Krylov83} and Schauder theory,
% we obtain Theorem \ref{mainthm1-de} as a complement of the following theorem
 %\ref{mainthm1-20172019}.
%proved in \cite{yuan2017,yuan2019}
% (see the final section of \cite{yuan2017} for $\eta=0$ and Remark 7.5 of \cite{yuan2019}).
%since the \textit{a priori} estimates are independent of $(\delta_{\psi,f})^{-1}$.

 \begin{theorem}
%Let $\partial M$ be smooth and \textit{$1$-concave},
[\cite{yuan2017,yuan2019}]
\label{mainthm1-20172019}
Let $(M,g)$ be a compact Riemannian manifold with smooth \textit{concave} boundary.
%$\chi$ and $\zeta$ be  smooth $(0,2)$ symmetric tensor and $(0,1)$-tensor on $\bar M$, respectively.
Let $\eta_{ij}^k=\delta_{ik}\zeta_j+\delta_{jk}\zeta_i$.
% Let $\mathfrak{g}[u]=\chi+\nabla^2u+d u\otimes\eta+\eta\otimes d u$.
In addition to \eqref{elliptic}, \eqref{concave}, \eqref{addistruc},  $\varphi\in C^{2,1}(\partial M)$,
 $\psi\in C^{1,1}(\bar M)$,   $f\in C^\infty(\Gamma)\cap C(\overline{\Gamma})$ and $\inf_M\psi=\sup_{\partial \Gamma}f$, we further
assume that there is a strictly admissible subsolution $\underline{u}\in C^{2,1}(\bar M)$ %and a $\delta_0>0$
   with $f(\lambda(\mathfrak{g}[\underline{u}])) \geq  \psi +\delta_0  \mbox{ in } \bar M, \mbox{   } \underline{u}|_{\partial M}= \varphi.$
   Then %the equation %Dirichlet problem %\eqref{mainequ20172019}-\eqref{mainequ1}
   \eqref{mainequ20172019}
   %and with boundary data \eqref{mainequ1}
admits a $C^{1,1}$ (weak) solution $u$ %$u\in C^{1,1}(\bar M)$
with $u|_{\partial M}=\varphi$, $\lambda(\mathfrak{g}[u])\in \overline{\Gamma}$
   and $\Delta u\in L^\infty(\bar M)$.

\end{theorem}

\section{The Dirichlet problem for degenerate equations on certain K\"ahler manifolds}
%\section{The Dirichlet problem for degenerate equations on K\"ahler manifolds with nonnegative orthogonal bisectional curvature}
\label{Kahlercase}

%In this section we assume $(M,\omega)$ is a compact K\"ahler manifold of complex dimension $n$
%with nonnegative orthogonal bisectional curvature. On such manifolds
By assuming that $(M,\omega)$ is a compact K\"ahler manifold %of complex dimension $n$
with nonnegative orthogonal bisectional curvature and the existence of subsolutions,
the author proved %second estimate
\eqref{2-seglobal} and gradient estimate in %Theorems 4.9 and 6.10 of
  \cite{yuan2019} (the first two versions), and solved %for the equation
\begin{equation}
\label{mainequ-Kahler}
\begin{aligned}
f(\lambda(\mathfrak{g}[u]))= \psi, %\mbox{ in  } M.
\mbox{   }\mathfrak{g}[u] =\chi+\sqrt{-1}\partial \overline{\partial} u+\sqrt{-1} \partial u\wedge  \overline{\eta^{1,0}}+\sqrt{-1}  \eta^{1,0} \wedge  \overline{\partial} u,
\end{aligned}
\end{equation}
with possibly degenerate right hand side when $\partial M$ is \textit{pseudoconcave} of
 $\mathcal{L}_{\partial M}\leq 0$. Here  ${\chi}$ is a smooth real $(1,1)$-form and
   $\eta^{1,0} =\eta_{i}d z^i$  is a smooth $(1,0)$-form.
%directly without using second order estimate.

In this section, with %a much more general assumption of
replaced $\mathcal{L}_{\partial M}\leq 0$ by $\mathrm{tr}_{\omega'}(\mathcal{L}_{\partial M})\leq 0$ (for simplicity, as in \cite{yuan2019}, we call it mean pseudoconcave) where
 $\omega'=\omega|_{T_{\partial M}\cap JT_{\partial M}}$,
we study the following
equation possibly with degenerate right hand side %of the form
\begin{equation}
\label{mainequ0-Kahler}
\begin{aligned}
 f(\lambda(W[u])) =\psi \mbox{ in } M,
\end{aligned}
\end{equation}
% with the boundary value condition
%\begin{equation}\label{mainequ1}\begin{aligned}
%u=  \varphi    \mbox{ on }\partial M,
%\end{aligned}\end{equation}
where $W[u]=\tilde{\chi}+ (\Delta u)\omega-\sqrt{-1}\partial\overline{\partial} u +Z(\partial u,\overline{\partial}u)$,
%$V=\mathrm{tr}_\omega \chi-\chi$, %
$\tilde{\chi}$ is a smooth real $(1,1)$-form,
$Z(\partial u,\overline{\partial}u)=\mathrm{tr}_\omega (\sqrt{-1} \partial u\wedge  \overline{\eta^{1,0}}+\sqrt{-1}  \eta^{1,0} \wedge  \overline{\partial} u)\omega-(\sqrt{-1} \partial u\wedge  \overline{\eta^{1,0}}
+\sqrt{-1}  \eta^{1,0} \wedge  \overline{\partial} u)$.
%we suppose $(M,\omega)$ is such a K\"ahler manifold with %further assuming
 %$\mathrm{tr}(\mathcal{L}_{\partial M})\leq 0$.

%This type of equations has been studied in an attempt to extend Calabi-Yau theorem to the counterpart of Gauduchon and balanced metrics.
% More precisely,
An interesting equation of this type
   is a Monge-Amp\`ere equation for $(n-1)$-plurisubharmonic functions %($f=\log\sigma_n$ and $\eta^{1,0}=0$)
 that is connected to the Gauduchon conjecture on closed astheno-K\"ahler manifolds first studied by Jost-Yau \cite{Jost1993Yau} (i.e. $\partial\overline{\partial}(\omega^{n-2})=0$, and complex surfaces are all astheno-K\"ahler).
 On such manifolds the Gauduchon conjecture was proved
 by Cherrier \cite{Cherrier} for $n=2$
 and later by  Tosatti-Weinkove \cite{Tosatti2013Weinkove} for higher dimensions, %$n\geq 3$,
 while for the Gauduchon conjecture on arbitrary closed Hermitian manifolds without carrying astheno-K\"ahler metric,
 the corresponding equation %on general Hermitian manifolds
 is much more %complicated
 hard to handle. %as it involves gradient terms
 For more references on Form-type Calabi-Yau equation and Monge-Amp\`ere equation for $(n-1)$-PSH functions, please refer to
 \cite{FuWangWuFormtype2010,FuWangWuFormtype2015,Tosatti2017Weinkove,GTW15,Popovici2015,guan-nie},
 %for the progress and follow-up work,
 and also to %the final section of
  \cite{yuan2019} for  Dirichlet problem on compact Hermitian manifolds with holomorphic flat boundary.
 %Besides, it is also related to Form-type Calabi-Yau equation on K\"ahler manifolds \cite{FuWangWuFormtype2015}.

 The following theorem concludes the existence of weak solutions to Dirichlet problems
  of degenerate equation \eqref{mainequ0-Kahler} on certain K\"ahler manifolds. % with nonnegative orthogonal bisectional curvature.
  % Monge-Amp\`ere equation for $(n-1)$-plurisubharmonic functions and Form-type Calabi-Yau equation on K\"ahler manifolds.

\begin{theorem}
Let $(M,\omega)$ be a compact K\"ahler manifold %of complex dimension $n$
with nonnegative orthogonal bisectional curvature and with smooth boundary of $\mathrm{tr}_{\omega'}(\mathcal{L}_{\partial M})\leq 0$.
 In addition to \eqref{elliptic}, \eqref{concave} and \eqref{addistruc}, we
assume $\varphi\in C^{2,1}(\partial M)$,  $\psi\in C^{1,1}(\bar M)$,
 $f\in C^\infty(\Gamma)\cap C(\overline{\Gamma})$ and $\inf_M\psi=\sup_{\partial \Gamma}f$.
Then equation \eqref{mainequ0-Kahler} supposes a  (weak)
   solution $u\in C^{1,\alpha}(\bar M)$ $(\forall 0<\alpha<1)$ with $u|_{\partial M}=\varphi$, $\lambda(W[u])\in \overline{\Gamma}$
   and $\Delta u\in L^\infty(\bar M)$, provided that
there is a strictly admissible subsolution $\underline{u}\in C^{2,1}(\bar M)$.
Moreover, if $\mathrm{tr}_{\omega'}(\mathcal{L}_{\partial M})< 0$ then the above statement is still
true when we assume $\partial M\in C^{2,1}$.
\end{theorem}

% \noindent{\bf Acknowledgements.}
% The author would like to express his gratitude to Prof. Bo Guan, Prof. Chunhui Qiu and Prof. Xi Zhang for their encouragement and support.
% The author is supported by NSF in China, grant 11801587.
%%%The author is partially supported by the National Natural Science Foundation of China (Grant No. 11801587).
%%%NSFC (Grant No. 11801587).

\begin{appendix}

\section{Construction of subsolutions}
  \label{appendix}
  In what follows we assume $\eta=0$ and $\eta^{1,0}=0$ for simplicity.
  The existence of subsolutions is required   according to above theorems.
Inspired by an idea of \cite{yuan2019}, on certain topologically product spaces, we can construct strictly admissible subsolutions
%$\frac{\partial }{\partial \nu}\underline{u}|_{\partial M}<0$
 for some equations  \eqref{mainequ0} and more general \eqref{mainequ20172019}.
 %under certain assumptions.
%As defined in Theorem \ref{proposition-quar-yuan1-thm},  $\nu$ is the unit inner normal vector along the boundary.

\vspace{1mm}
%\subsection{Riemannian case}

\noindent{\bf Real variables}:

\vspace{1mm}
  $\bullet$ {\bf Case I}: $(M,g)$ is a warped product space $(X\times (0,1), e^\varrho g_X+ dx^n\otimes dx^n)$ for $\varrho\in C^\infty(\bar M)$, ($\bar M=X\times [0,1]$).

   \begin{itemize}
   \item
  For equation \eqref{mainequ20172019}:
Suppose there is an admissible function $\underline{w}$ with $\lambda(\mathfrak{g}[\underline{w}])\in\Gamma$ such that
\begin{equation}
\begin{aligned}
\lim_{t\rightarrow+\infty}f(\lambda(\mathfrak{g}[\underline{w}]+tdx^n\otimes dx^n)) > \psi  \mbox{ in } \bar M, \mbox{   }
\underline{w}= \varphi   \mbox{ on } \partial  M  \nonumber
\end{aligned}
\end{equation}
which is automatically satisfied if $f$ further obeys the unbounded condition %\eqref{unbound}.
\begin{equation}\label{unbound} \begin{aligned}
\lim_{t\rightarrow+\infty}f(\lambda_1,\cdots,\lambda_n+t)=\sup_\Gamma f, \mbox{  } \forall \lambda=(\lambda_1,\cdots,\lambda_n)\in \Gamma.
\end{aligned}\end{equation}
%where $e_i$ is the $i$-$\mathrm{th}$ standard basis vector.
Then  we can construct subsolutions for  \eqref{mainequ20172019} on such warped product spaces, and the subsolution is given by $\underline{u}=\underline{w}+A(x_n^2-x_n)$ for $A\gg1$.
%Notice that $(M,g)=(X\times [0,1],e^hg_X+e^\rho dx^n\otimes dx^n)$ admits a \textit{concave} boundary if $h, \rho$ are smooth functions on $M$ with $\nabla_\nu h|_{\partial M}\geq0$,

\item For equation \eqref{mainequ0}: Similarly, on such warped product spaces, we can construct subsolutions, if
%and there is an admissible function $\underline{w}$ with $\lambda(U[\underline{w}])\in\Gamma$ such that
\begin{equation}
\begin{aligned}
\lim_{t\rightarrow+\infty}f(\lambda(\mathfrak{g}[\underline{w}]+tg_X)) > \psi  \mbox{ in } \bar M, \mbox{   }
\underline{w}= \varphi   \mbox{ on } \partial  M  \nonumber
\end{aligned}
\end{equation}
holds for an admissible function $\underline{w}$ with $\lambda(U[\underline{w}])\in\Gamma$. %Such a condition leads to
% \begin{equation}\label{unbound-strong} \begin{aligned}\lim_{t\rightarrow+\infty}f(\lambda_1+t,\lambda_2+t, \cdots,\lambda_{n-1}+t, \lambda_n)=\sup_\Gamma f, \mbox{  } \forall \lambda=(\lambda_1,\cdots,\lambda_n)\in \Gamma.\end{aligned}\end{equation}

\end{itemize}

 $\bullet$  {\bf Case II}: $(M,g)=(X\times \Omega, g)$ is  a product of $(n-k)$-dimensional closed Riemannian manifold   $(X,g_X)$ with
  a bounded smooth domain %$(\Omega,g_\Omega)$,
  $\Omega\subset \mathbb{R}^k$, $2\leq k\leq n$.
Also, we denote $g_\Omega = \sum_{j=n-k+1}^ndx^j\otimes d x^j$ and
 $$\Gamma^\infty_{\mathbb{R}^1}=\{c\in \mathbb{R}: (t,\cdots,t,c)\in\Gamma \mbox{ for } t\gg1\}.$$
 In particular, if $c>0$ then $c\in \Gamma^\infty_{\mathbb{R}^1}$.
% $\Gamma^\infty(\mathbb{R}^k):=\{(x_1,\cdots,x_k)\in \mathbb{R}^k: (t,\cdots,t,x_1,\cdots,x_k)\in\Gamma \mbox{ for } t\gg1\}.$

\begin{itemize}
\item For equation \eqref{mainequ20172019}: Suppose furthermore that $\Omega$ is a strictly convex domain.
If there is an admissible function $\underline{w}$ such that
 \begin{equation}
\begin{aligned}
\lim_{t\rightarrow+\infty}f(\mathfrak{g}[\underline{w}]+tg_\Omega)) > \psi  \mbox{ in } \bar M, \mbox{   }
\underline{w}= \varphi   \mbox{ on } \partial  M, \nonumber
\end{aligned}
\end{equation}
then
 the %strictly admissible
 subsolution is given by $\underline{u}=\underline{w}+Ah \mbox{ for large } A,$
 where $h$ is a smooth strictly convex function with $h|_{\partial \Omega}=0$.

\item For equation \eqref{mainequ0}:  We assume
$\mathrm{H}_{\partial \Omega}\in \Gamma^\infty_{\mathbb{R}^1}$, $g=e^\varrho g_X+ g_\Omega$, $\varrho \in C^\infty(\bar M)$, and $f$ further satisfies \eqref{addistruc}.

A somewhat interesting fact is that under such assumptions, we can construct subsolutions with arbitrary $\varphi\in C^2(\partial M)$, since there exists a function
$h\in C^\infty(\bar \Omega)$ with $h|_{\partial \Omega}=0$ such that $\lambda(\Delta  h g -\nabla^2 h)\in \Gamma$ in $\bar M$.
%$\lambda(\Delta_\Omega h g_\Omega -\nabla^2 h)\in \Gamma^\infty(\mathbb{R}^k)$ in $\bar \Omega$.

Note that   $\mathbb{R}^+\subseteq  \Gamma^\infty_{\mathbb{R}^1}$.
%In particular,
 So if $\Omega$ is a strictly mean convex domain, %$\mathrm{H}_{\partial \Omega}>0$,
then  $\mathrm{H}_{\partial \Omega}\in \Gamma^\infty_{\mathbb{R}^1}$, and  \eqref{mainequ0} admits a smooth subsolution on such background manifolds.

%A somewhat interesting fact is that for each $\varphi\in C^2(\partial M)$, one always has an extension $\underline{w}$ being admissible.
%Interestingly, $\mathrm{H}_{\partial \Omega}\in \Gamma^\infty(\mathbb{R}^1)$  for each strictly mean convex domain.
 \end{itemize}

\vspace{1mm}
\noindent{\bf Complex variables}:
  $M=X\times N$
  is a product of a closed complex manifold $(X,\omega_X)$ of complex dimension  $(n-k)$ with
  a compact complex manifold $(N,\omega_N)$ of complex dimension  $k$ with boundary. % and $\eta_{N}^{1,0}$ is $(1,0)$-form on $N$.
When $N$ is a compact Riemannian surface with boundary, i.e. $k=1$, this case was already considered by the author in \cite{yuan2019}.

%We suppose now that $k\geq 2$ and $N=\Omega\subset\mathbb{C}^k$ is a bounded $C^{l,\alpha}$ strictly pseudoconvex domain, as above $2\leq k\leq n$, $l\geq2$ and $0<\alpha<1$.
%  Let $h$ be a $C^{l,\alpha}$-smooth strictly pseudoconvex function
%   with $h|_{\partial \Omega}=0$, $\omega_\Omega =\sqrt{-1}\sum_{j=n-k+1}^ndz^j\wedge d\bar z^j$ denote
%the standard metrics of $\mathbb{C}^k$.

Suppose now that $N=\Omega\subset\mathbb{C}^k$ is a bounded smooth  domain, as above $2\leq k\leq n$.
  Let  $\omega_\Omega =\sqrt{-1}\sum_{j=n-k+1}^ndz^j\wedge d\bar z^j$ denote
the standard metric  of $\mathbb{C}^k$.

 Similar to  Riemannian case, we can construct strictly admissible subsolutions.
More precisely,

\begin{itemize}
\item For \eqref{mainequ-Kahler}:  Let  $\Omega$ be a  strictly pseudoconvex domain,
$h$ be a  smooth strictly pseudoconvex function
   with $h|_{\partial \Omega}=0$.
   The subsolution is given by $\underline{u}=\underline{w}+Ah \mbox{ for large } A,$ provided there is an admissible function $\underline{w}$ such that
  \begin{equation}
\begin{aligned}
\lim_{t\rightarrow+\infty}f(\mathfrak{g}[\underline{w}]+t\omega_\Omega)) > \psi  \mbox{ in } \bar M, \mbox{   }
\underline{w}= \varphi   \mbox{ on } \partial  M.  \nonumber
\end{aligned}
\end{equation}
\item For \eqref{mainequ0-Kahler}: We assume  $\mathrm{tr} (\mathcal{L}_{\partial \Omega})\in\Gamma^\infty_{\mathbb{R}^1}$,
$\omega=e^\varrho\omega_X+e^\rho\omega_\Omega$, for $\varrho$, $\rho\in C^\infty(\bar M)$. Similarity,  we can construct subsolutions with arbitrary $\varphi\in C^2(\partial M)$, if $f$ further satisfies \eqref{addistruc}.
\end{itemize}

\section{A remark on quantitative boundary estimate}
\label{append-B}

In this appendix we further extend the quantitative boundary estimates (Theorem \ref{quan-boundaryestimate-thm1})
from    mean concave  boundary to general boundary. 
%To do this we further assume there is   admissible supersolutions for \eqref{mainequ0}-\eqref{mainequ1}:
%\begin{equation}
%\label{supersolution1}
%\begin{aligned}
%F(\mathfrak{g}[\breve{u}])\leq \psi  \mbox{ and } \lambda(\mathfrak{g}[\breve{u}])\in \Gamma \mbox{ in } \bar M, \mbox{   }  \breve{u}=\varphi \mbox{ on } \partial M.
%\end{aligned}
%\end{equation}
More precisely,

 \begin{theorem}

  Let $(M,g)$ be a compact Riemannian manifold with   smooth general
   boundary.
Let $\psi\in C^{1}(\bar M)$, $\varphi\in C^{3}(\partial M)$.
Suppose, in addition to  \eqref{elliptic}, \eqref{concave}, \eqref{nondegenerate} and \eqref{existenceofsubsolution}, that $f$ satisfies \eqref{unbound-strong} and
there is an admissible supersolution  $\breve{u}$ with %for \eqref{mainequ0}-\eqref{mainequ1}
\begin{equation}
\label{supersolution1}
\begin{aligned}
F(U[\breve{u}])\leq \psi  \mbox{ and } \lambda(\mathfrak{g}[\breve{u}])\in \Gamma \mbox{ in } \bar M, \mbox{   }  \breve{u}=\varphi \mbox{ on } \partial M.
\end{aligned}
\end{equation}
Then for any admissible solution $u\in C^3(M)\cap C^2(\bar M)$ to   Dirichlet problem
 \eqref{mainequ0}-\eqref{mainequ1},  quantitative boundary estimate \eqref{good-quard} holds for
%  \begin{equation}
% %\label{good-quard}
% \begin{aligned}
% \sup_{\partial M} \Delta u \leq C %(1+\sup_{\partial M}|\nabla u|^2)
% (1+\sup_M |\nabla u|^2),
% \end{aligned}
% \end{equation}
  a uniformly positive constant  that depends on $|\varphi|_{C^{3}(\bar M)}$, $|\psi|_{C^{1}(\bar M)}$,
$|\underline{u}|_{C^{2}(\bar M)}$, $|\breve{u}|_{C^2(\bar M)}$, $\partial M$
up to its third derivatives
and other known data.

 \end{theorem}

 \begin{proof}[Sketch of proof]
 %Similar to \eqref{opppp},
 Condition \eqref{unbound-strong} implies that there exist two uniformly positive constants $\varepsilon_{0}'$, $R_{0}'$
 depending  on  $\breve{\mathfrak{g}}$  and $f$, such that
\begin{equation}
\label{opppp*}
\begin{aligned}
\,& f(R_0'-\underline{\lambda}'_1-\varepsilon_{0}',\cdots, R_0'-\underline{\lambda}'_{n-1}-\varepsilon_{0}',\sum_{\alpha=1}^{n-1}\breve{\mathfrak{g}}_{\alpha\alpha}-\varepsilon_{0}')\geq \psi
%\mbox{ and }\,& (\mathfrak{\underline{g}}_{1\bar 1}-\varepsilon_{0}, \cdots, \underline{\mathfrak{g}}_{(n-1)\overline{(n-1)}}-\varepsilon_0, R_1) \in \Gamma.
\end{aligned}
%\right.
\end{equation}
  In the case of $\mathrm{H}_{\partial M}\geq 0$, in place of \eqref{opppp} we  use \eqref{opppp*}, where we use  ${(\breve{u}-u)_{\nu}}|_{\partial M}\geq 0$.

 \end{proof}

 \end{appendix}

%\begin{appendix}
% \end{appendix}

%\medskip

%\noindent

%\vspace{3mm}

\bigskip

\small
\bibliographystyle{plain}

\end{CJK*}

\end{document}